\documentclass[11pt]{amsart}

\usepackage{url}
\usepackage{graphicx,pinlabel}
\usepackage{palatino}
\usepackage{parskip}
\usepackage{amsthm,amsfonts,amsmath,amssymb,amscd}
\usepackage{verbatim}
\usepackage{enumerate}
\usepackage{mathtools,bm,graphicx}

\newtheorem{thm}{Theorem}[section]
\newtheorem{lma}[thm]{Lemma}
\newtheorem{cor}[thm]{Corollary}

\newtheorem{clm}{Claim}

\renewcommand{\bar}{\overline}
\renewcommand{\tilde}{\widetilde}

\theoremstyle{remark}
\newtheorem{rmk}[thm]{Remark}

\theoremstyle{definition}
\newtheorem{dfn}[thm]{Definition}
\newtheorem{exm}[thm]{Example}

\newcommand{\mb}[1]{\mathbf{#1}}
\newcommand{\CC}{\mb{C}}

\newcommand{\PP}{\mb{P}}
\newcommand{\QQ}{\mb{Q}}
\newcommand{\RR}{\mb{R}}
\newcommand{\ZZ}{\mb{Z}}

\newcommand{\cp}[1]{\CC\PP^{#1}}

\newcommand{\rt}{\bm{\mu}}

\newcommand{\TO}[3]{
#1\xrightarrow{#2} #3
}

\newcommand{\eE}{\mathcal{E}}

\newcommand{\orb}[1]{\skew{3}\hat{#1}}

\mathchardef\ordinarycolon\mathcode`\:
\mathcode`\:=\string"8000
\begingroup \catcode`\:=\active
  \gdef:{\mathrel{\mathop\ordinarycolon}}
\endgroup

\makeatletter
\newcommand*{\da@rightarrow}{\mathchar"0\hexnumber@\symAMSa 4B }
\newcommand*{\da@leftarrow}{\mathchar"0\hexnumber@\symAMSa 4C }
\newcommand*{\xdashrightarrow}[2][]{%
  \mathrel{%
    \mathpalette{\da@xarrow{#1}{#2}{}\da@rightarrow{\,}{}}{}%
  }%
}
\newcommand{\xdashleftarrow}[2][]{%
  \mathrel{%
    \mathpalette{\da@xarrow{#1}{#2}\da@leftarrow{}{}{\,}}{}%
  }%
}
\newcommand*{\da@xarrow}[7]{%
  \sbox0{$\ifx#7\scriptstyle\scriptscriptstyle\else\scriptstyle\fi#5#1#6\m@th$}%
  \sbox2{$\ifx#7\scriptstyle\scriptscriptstyle\else\scriptstyle\fi#5#2#6\m@th$}%
  \sbox4{$#7\dabar@\m@th$}%
  \dimen@=\wd0 %
  \ifdim\wd2 >\dimen@
    \dimen@=\wd2 %
  \fi
  \count@=2 %
  \def\da@bars{\dabar@\dabar@}%
  \@whiledim\count@\wd4<\dimen@\do{%
    \advance\count@\@ne
    \expandafter\def\expandafter\da@bars\expandafter{%
      \da@bars
      \dabar@ 
    }%
  }%
  \mathrel{#3}%
  \mathrel{%
    \mathop{\da@bars}\limits
    \ifx\\#1\\%
    \else
      _{\copy0}%
    \fi
    \ifx\\#2\\%
    \else
      ^{\copy2}%
    \fi
  }%
  \mathrel{#4}%
}
\makeatother

\title[Markov numbers and Lagrangian cell complexes]{Markov numbers and Lagrangian cell complexes\\ in the complex projective plane}
\author{Jonathan David Evans}
\address{Jonathan Evans\\ Department of Mathematics\\ University College London\\ Gower Street\\ London\\ WC1E 6BT\\ United Kingdom}
\email{j.d.evans@ucl.ac.uk}
\author{Ivan Smith}
\address{Ivan Smith\\ Centre for Mathematical Sciences\\ University of Cambridge\\ Wilberforce Road\\ CB3 0WB\\ United Kingdom.}
\email{is200@cam.ac.uk}

\dedicatory{\begin{center}To the European Union, with sincere regret.\end{center}}

\begin{document}

\begin{abstract}
  We study Lagrangian embeddings of a class of two-dim\-en\-s\-ion\-al cell complexes $L_{p,q}$ into the complex projective plane. These cell complexes, which we call {\em pinwheels}, arise naturally in algebraic geometry as vanishing cycles for quotient singularities of type $\frac{1}{p^2}(pq-1,1)$ (Wahl singularities). We show that if a pinwheel admits a Lagrangian embedding into $\cp{2}$ then $p$ is a Markov number and we completely characterise $q$. We also show that a collection of Lagrangian pinwheels $L_{p_i,q_i}$, $i=1,\ldots,N$, cannot be made disjoint unless $N\leq 3$ and the $p_i$ form part of a Markov triple. These results are the symplectic analogue of a theorem of Hacking and Prokhorov, which classifies complex surfaces with quotient singularities admitting a $\QQ$-Gorenstein smoothing whose general fibre is $\cp{2}$.
\end{abstract}

\maketitle

\section{Introduction}

\begin{dfn}
  Let $p_1,p_2,p_3$ be positive integers. The triple $(p_1,p_2,p_3)$ is a {\em Markov triple} if
  \begin{equation}\label{eq:markov}
    p_1^2+p_2^2+p_3^2=3p_1p_2p_3.
  \end{equation}
\end{dfn}

A theorem of Hacking and Prokhorov \cite{HackingProkhorovDP,HackingProkhorov} asserts that if $X$ is a projective algebraic surface with quotient singularities which admits a $\QQ$-Gorenstein smoothing to the complex projective plane $\cp{2}$, then $X$ is obtained by partially smoothing a weighted projective plane $\CC\PP(p_1^2,p_2^2,p_3^2)$ for a Markov triple $(p_1,p_2,p_3)$. In particular, $X$ has at most $3$ singular points (this was also known earlier to Manetti \cite{Manetti}). Koll\'ar \cite{Kollar} has asked to what extent such theorems, which put constraints on orbifold degenerations of algebraic surfaces, are purely topological. More precisely, he conjectures that any compact smooth 4-manifold $M$ with $H_1(M;\ZZ)=0$ and $H_2(M;\ZZ)=\ZZ$ can have at most five boundary components with finite but nontrivial fundamental group. The relation between Koll\'ar's conjecture and degenerations is that such a 4-manifold could be obtained by excising neighbourhoods of the singularities in a degeneration of a homology $\cp{2}$, so this conjecture would provide a topological upper bound on the number of orbifold singularities. This conjecture is motivated by the orbifold Bogomolov-Miyaoka-Yau inequality (which provides the upper bound on the number of singularities, but relies on some nontrivial algebraic geometry).

In this paper, we show that the Hacking-Prokhorov theorem has a purely symplectic analogue (Theorem \ref{thm:markov} below). This can be interpreted as giving constraints on the topology and displaceability of Lagrangian embeddings of certain simple cell complexes in $\cp{2}$. We do not know if this can be further weakened to use only methods of differential topology as Koll\'ar conjectures.

An interesting feature of the algebraic surfaces above is that the cyclic quotient singularities $\frac{1}{p^2}(1,pq-1)$ which appear have vanishing Milnor number; the local smoothing of the singularity is a rational homology ball $B_{p,q}$.  From a symplectic perspective, $B_{p,q}$ is naturally a Stein domain, and retracts to a Lagrangian skeleton $L_{p,q}$ which we call a \emph{pinwheel}, homotopy equivalent to the Moore space $M(\ZZ/p, 1)$.  (In the simplest example $B_{2,1} \cong T^*\RR\PP^2$ and $L_{2,1} \cong \RR\PP^2$.)  One can formulate potential symplectic counterparts to (strengthenings of) the Hacking-Prokhorov theorem  by asking for constraints on four-dimensional symplectic orbifolds which can be smoothed to $\PP^2$, or by asking about symplectic embeddings $\amalg_i \, B_{p_i,q_i} \subset \PP^2$, or embeddings of the corresponding Lagrangian skeleta. Symplectic embedding questions for rational balls were first considered by Khodorovskiy \cite{KhodorovskiyThesis,KhodorovskiyBounds,KhodorovskiyQBU}; see also \cite{LM}. 

\begin{thm}\label{thm:markov}
  Let $N$ be a positive integer and $B_{p_i,q_i}\subset\cp{2}$, $i=1,\ldots,N$, be a collection of symplectic embeddings having pairwise disjoint images. Equivalently, let $L_{p_i,q_i}\subset\cp{2}$ be a collection of pairwise disjoint Lagrangian pinwheels. Then $N\leq 3$ (Corollary \ref{cor:fourball}). Moreover:
  \begin{enumerate}
  \item[A.] (Theorem \ref{thm:oneball}) If $N=1$ then $p_1$ belongs to a Markov triple $(p_1,b,c)$. Moreover, $q_1=\pm 3b/c\mod p_1$.
  \item[B.] (Theorem \ref{thm:twoball}) If $N=2$ then $p_1$ and $p_2$ belong to a Markov triple $(p_1,p_2,c)$. Moreover, $q_1=\pm 3p_2/c\mod p_1$ and $q_2=\pm 3p_1/c\mod p_2$.
  \item[C.] (Theorem \ref{thm:threeball}) If $N=3$ then $(p_1,p_2,p_3)$ is a Markov triple. Moreover, $q_i=\pm 3p_j/p_k\mod p_i$ where $i,j,k$ is a permutation of $1,2,3$.
  \end{enumerate}
\end{thm}

\begin{cor} We have the following consequences for Lagrangian embeddings.
  \begin{enumerate}
  \item[A.] If $p$ is not a Markov number, or if $p$ is Markov but $q\neq \pm 3b/c$ for any $b,c<p$ with $p^2+b^2+c^2=3pbc$, then there is no Lagrangian embedding of $L_{p,q}$ into $\cp{2}$. In particular, if $q^2\neq -9\mod p$ then there is no Lagrangian embedding of $L_{p,q}$ into $\cp{2}$.
  \item[B.] If $p$ and $p'$ are Markov numbers but do not form part of a Markov triple, and $L_{p,q}$ and $L_{p',q'}$ are Lagrangian pinwheels in $\cp{2}$ then they cannot be disjoined from one another by a Hamiltonian isotopy.
  \end{enumerate}
\end{cor}
\begin{proof}
  This is immediate from Theorem \ref{thm:markov} A and B. The only nontrivial observation is that if $q=\pm 3b/c$ for some Markov triple $(p,b,c)$ then $q^2=-9\mod p$ {\cite[Chapter I.3, Equation (6)]{Cassels}}.
\end{proof}

\begin{rmk}
  The first few Markov triples are:
  \begin{gather*}(1,1,1),\ (1,1,2),\ (1,2,5),\ (1,5,13),\ (2,5,29),\\ (1,13,34),\ (5,13,194),\ (5,29,433),\ (2,29,169),\ldots\end{gather*}
  A number appearing in a Markov triple is called a {\em Markov number}; see \cite{OEIS} for more Markov numbers. Markov numbers share some arithmetic properties. For example, since $p_1^2+p_2^2=(3p_1p_2-p_3)p_3$, $p_3$ is a factor of a sum of squares; therefore all odd prime factors of a Markov number are congruent to $1$ modulo $4$. In particular, Theorem \ref{thm:markov} implies that $B_{3,q}$, $B_{7,q}$, $B_{11,q}\ldots$ never embed symplectically in $\cp{2}$.
\end{rmk}

\begin{rmk}
  For every Markov triple $(p_1,p_2,p_3)$ there exist three disjoint Lagrangian pinwheels $L_{p_i,q_i}\subset\cp{2}$, which are the vanishing cycles of the $\QQ$-Gorenstein degeneration to $\CC\PP(p_1^2,p_2^2,p_3^2)$ so the result is sharp. This gives many pairs of Lagrangian pinwheels which cannot be displaced from one another. Since these Lagrangians are not even immersed submanifolds it is not clear how to prove this nondisplaceability result using classical Floer theory.
\end{rmk}

\begin{rmk}
  The sign ambiguity in $q_i=\pm 3p_j/p_k\mod p_i$ is not really an ambiguity at all, as $B_{p,q}$ and $B_{p,p-q}$ are symplectomorphic; see Remark \ref{rmk:signamb}.
\end{rmk}

\begin{rmk}
  The equivalence of the statements about Lagrangian pinwheels and symplectic rational homology balls follows from a neighbourhood theorem due to Khodorovskiy, see Section \ref{sct:khodorovskiy} below.
\end{rmk}

\begin{rmk}
  The exotic monotone Lagrangian tori discovered by Vianna \cite{Vianna1,Vianna} are also in bijection with Markov triples; they are constructed by taking the barycentric torus in a toric weighted projective plane $\CC\PP(p_1^2,p_2^2,p_3^2)$ and transporting it to a nearby smooth fibre in the $\QQ$-Gorenstein degeneration (as such, they are disjoint from the pinwheel vanishing cycles). Indeed, the neck-stretching arguments used to constrain the superpotential in Vianna's paper \cite{Vianna} were important inspiration for the current paper.
\end{rmk}

\begin{rmk}
  It would be interesting to investigate this relationship further for Del Pezzo surfaces with $b_2>1$. On the one hand, Vianna's work \cite{ViannaDP} gives a plethora of almost toric structures corresponding to solutions of Diophantine equations, where there are visible embedded pinwheels; on the other hand, the literature on the related problem of classifying helices on del Pezzo surfaces suggests that $\cp{2}$ and its relation to Markov triples is a special case \cite{KarpovNogin}.
\end{rmk}

The idea of the proof is as follows. Given a rational ball $B_{p,q} \subset \cp{2}$, one can stretch the neck (in the sense of symplectic field theory) around $\partial B_{p,q}$ and study the limits of holomorphic lines in $\cp{2}$. Those limits compactify to give holomorphic curves in orbifold degenerations of $\cp{2}$, and our results are essentially applications of Weimin Chen's orbifold adjunction formula to these curves, in the spirit of W.~Chen's own work \cite{ChenOrbiAdj,ChenPseudo}, together with a number-theoretic result on Markov-type Diophantine equations due to Rosenberger \cite{Rosenberger79}. 

\subsection*{Outline}

In Section \ref{sct:pin}, we review some basic topological facts about Lagrangian pinwheels $L_{p,q}$ and about the symplectic rational homology balls $B_{p,q}$.

In Section \ref{sct:orbi}, we explain how one can replace a symplectic rational homology ball $B_{p,q}\subset X$ by an orbifold singularity. We also review W. Chen's theory of holomorphic curves in orbifolds, including the adjunction formula and the formula for the virtual dimension of the moduli space of orbifold curves. These are our main tools in what follows.

Section \ref{sct:orbicurve} is the bulk of the proof. In Section \ref{sct:lowdeg}, we prove that if a symplectic orbifold $\orb{X}$ arises from $X=\cp{2}$ by collapsing $N$ symplectically embedded rational homology balls $B_{p_i,q_i}$, $i=1,\ldots,N$, then there are strong constraints on the orbifold curves $\orb{C}$ of low degree in $\orb{X}$. In particular, if the degree is less than $\Delta:=\prod_{i=1}^Np_i$ then we show that if $Z\subset\orb{C}$ is the set of orbifold points then $1\leq |Z|\leq 2$. In Section \ref{sct:sftanalysis}, we show that (a) there are always curves with $|Z|=1$ and positive virtual dimension, and (b) if $\orb{X}$ has multiple singularities, then there are also curves with $|Z|=2$ which meet particular pairs of orbifold points.

The $|Z|=1$ curves allow us to prove Theorem \ref{thm:markov} A. We will show that the adjunction formula for such curves becomes the Markov equation, where the numbers $b$ and $c$ have the following geometric interpretation: in a local lift of the curve to $\CC^2$ in a neighbourhood of the orbifold point $\CC^2/\Gamma$, the link of the orbifold point of the curve is a $(b^2,c^2)$-torus knot.

The $|Z|=2$ curves allow us to prove Theorem \ref{thm:markov} B. We will show that these curves are suborbifolds, and again the adjunction formula becomes the Markov equation.

Theorem \ref{thm:markov} C and the fact that $N\leq 3$ (Corollary \ref{cor:fourball}) follow readily from A and B and elementary properties of Markov triples.

\section{Pinwheels and rational homology balls}\label{sct:pin}

\subsection{Lagrangian pinwheels and Khodorovskiy neighbourhoods}\label{sct:khodorovskiy}

\begin{dfn}[Pinwheel]
  Let $D$ denote the unit disc in $\CC$ and let $\sim_m$ denote the equivalence relation on $D$ which identifies $z$ and $z'$ if $z,z'\in\partial D$ and $(z/z')^m=1$. The quotient space $D/{\sim_m}$ is a CW-complex which we call the {\em pinwheel} $P_m$. The image of $\partial D$ in $P_m$ is called the {\em core circle}. See Figure \ref{fig:core} for an illustration of a neighbourhood of an arc in the core circle in $P_5$.
\end{dfn}
\begin{rmk}
  The pinwheel $P_m$ is a Moore space $M(\ZZ/(m),1)$, that is its reduced integral homology groups are
  \[\tilde{H}_k(P_m,\ZZ)=\begin{cases}
  \ZZ/(m)&\mbox{ if }k=1\\
  0&\mbox{ otherwise.}
  \end{cases}\]
\end{rmk}

\begin{dfn}[Lagrangian pinwheel; see {\cite[Definition 3.1]{KhodorovskiyQBU}}]
  Let $p$ be a positive integer and $0<q<p$ an integer coprime to $p$. A Lagrangian $(p,q)$-pinwheel in a symplectic manifold $(X,\omega)$ is a smooth Lagrangian immersion $f\colon D\looparrowright X$ such that:
  \begin{itemize}
  \item $f|_{D\setminus\partial D}$ is an embedding;
  \item $f(x)=f(y)$ if and only if $x,y\in\partial D$ with $x\sim_py$ (in other words, $f$ factors through a continuous embedding $P_p\to X$);
  \item if $f(x)=f(y)$ then $f_*(T_xD)\neq f_*(T_yD)$.
  \end{itemize}
  The number $q$ can be characterised as follows. Let $\theta$ be the angle coordinate on $D\setminus\{0\}$ and let $\Lambda\to D\setminus\{0\}$ be the $S^1$-bundle whose fibre over $x$ is the space of Lagrangian 2-planes in $T_{f(x)}X$ which contain $f_*(\partial_\theta)$. Let $S=f(\partial D)$ be the core of the pinwheel. Since $f(x)=f(y)$ if $x\sim_py$ for $x,y\in\partial D$, the bundle $\Lambda|_{\partial D}$ restricted to the boundary of the disc is the pullback of a bundle $\Lambda'\to S$ under the $p$-fold covering map $\partial D\to S$. Pick a trivialisation of $\Lambda'$ and extend it to a trivialisation of $\Lambda$. We have a section $\sigma\colon\partial D\to\Lambda$ defined by $\sigma(x)=f_*(T_xD)$; this section has a winding number $q$ relative to the chosen trivialisation. If we change the trivialisation of $\Lambda'$ and extend this to a new trivialisation of $\Lambda$ then we change $q$ by a multiple of $p$. Hence the residue modulo $p$ of the winding number is well-defined. See Lemma \ref{lma:firstchernclass} for a related characterisation of $q$.
\end{dfn}

\begin{figure}[htb]
  \begin{center}
    \includegraphics[width=200pt]{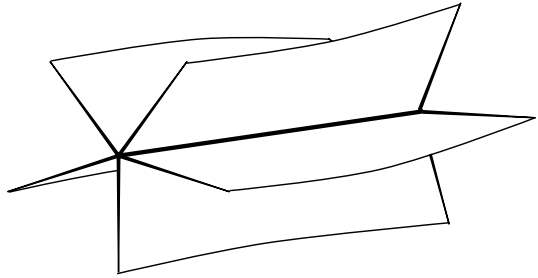}
  \end{center}
  \caption{A neighbourhood of an arc in the core circle in the pinwheel $P_5$.}
  \label{fig:core}
\end{figure}

\begin{dfn}
  Let $n$ be a positive integer and let $\rt_n=\{\zeta\in\CC\ :\ \zeta^n=1\}$ denote the group of complex $n$th roots of unity. Given integer weights $m_1,\ldots,m_k$, we define the quotient singularity
  \[\frac{1}{n}(m_1,\ldots,m_k)\]
  to be the singularity at $0$ of the quotient of $\CC^k$ by the action of $\rt_n$ with weights $m_1,\ldots,m_k$:
  \[(z_1,\ldots,z_k)\mapsto (\zeta^{m_1}z_1,\ldots,\zeta^{m_k}z_k).\]
\end{dfn}

\begin{exm}
  Let $p,q$ be coprime integers with $1\leq q<p$. Let $\Gamma_{p,q}$ be the action of $\rt_{p^2}$ with weights $(1,pq-1)$. The surface $\CC^2/\Gamma_{p,q}$ has a singularity of type $\tfrac{1}{p^2}(1,pq-1)$ at the origin. Consider the 3-fold $\CC^3/\rt_p$ of type $\frac{1}{p}(1,-1,q)$. We can embed the surface $\CC^2/\Gamma_{p,q}$ into this 3-fold as the subvariety $\{xy=z^p\}$. This allows us to find a $\QQ$-Gorenstein smoothing
  \[\mathcal{S}_{p,q}=\{xy=z^p+t\}\subset\CC^3/\rt_p\times\CC_t\]
  of $\CC^2/\Gamma_{p,q}$. Pick $t\neq 0$ and let $S_{p,q}$ be the fibre of this $\QQ$-Gorenstein smoothing over $t$. This is the Milnor fibre of $\CC^2/\Gamma_{p,q}$. Recall {\cite[Example 5.9.1]{Wahl}} that the singularity has vanishing Milnor number, so $S_{p,q}$ is a rational homology ball.

The variety $S_{p,q}$ contains a Lagrangian $(p,q)$-pinwheel $L_{p,q}$. To see this, we use a particular presentation of $S_{p,q}$ as a quotient of a Lefschetz fibration, described in \cite{LM}. Consider the $A_{p-1}$-Milnor fibre
  \[A_{p-1}=\{(x,y,z)\ :\ z^p+2xy=1\}\]
  and take its quotient by the (free) $\rt_p$-action where $\zeta\in\rt_p$ acts by
  \[\zeta\cdot(x,y,z)=(\zeta x,\zeta^{-1}y,\zeta^qz).\]
  The Lefschetz fibration $\pi\colon A_{p-1}\to\CC$, $\pi(x,y,z)=z$, is a conic fibration with $p$ nodal fibres. Let $\gamma_1,\ldots,\gamma_p$ be straight line segments connecting the critical values to the origin. The union of the vanishing thimbles over these paths is the universal cover of $P_p$ and descends to give a Lagrangian $(p,q)$-pinwheel $L_{p,q}$ in the quotient.

  The variety $S_{p,q}$ admits a Liouville form for which it is the symplectic completion of a compact Stein domain $B_{p,q}$ and for which $L_{p,q}$ is the Lagrangian skeleton in the sense of \cite{Biran}. We will not distinguish between different possible choices of subdomain $B_{p,q}$ since any two share a common Liouville retract and we are most interested in the pinwheel itself.
\end{exm}

\begin{exm}
  By a local-to-global theorem of Hacking and Prokhorov {\cite[Proposition 3.1]{HackingProkhorov}}, there is a $\QQ$-Gorenstein smoothing of any Fano surface with log canonical singularities, provided the singularities can locally be smoothed. They show that if $(p_1,p_2,p_3)$ is a Markov triple then the smoothing of $\mathbf{CP}(p_1^2,p_2^2,p_3^2)$ is a Fano surface with $K^2=9$ and hence biholomorphic to $\cp{2}$. Here is one way to see this, using the fact that $K^2$ is constant in $\QQ$-Gorenstein families.
  
  Suppose that $(p_1,p_2,p_3)$ is a Markov triple and let $p'_3=3p_1p_2-p_3$. The triple $(p_1,p_2,p'_3)$ is again a Markov triple and we call the transition $(p_1,p_2,p_3)\to(p_1,p_2,p'_3)$ a {\em mutation}. The $\QQ$-Gorenstein deformation
  \[\{z_0z_1-(1-t)z_2^{p'_3}-tz_3^{p_3}=0\}\subset\CC\PP(p_1^2,p_2^2,p_3,p'_3)\]
  connects $\CC\PP(p_1^2,p_2^2,p_3^2)$ at $t=0$ to $\CC\PP(p_1^2,p_2^2,(p'_3)^2)$ at $t=1$, so these both have the same value of $K^2$. Any Markov triple $(p_1,p_2,p_3)$ can be related to $(1,1,1)$ by a sequence of mutations, and $\mathbf{CP}(1,1,1)=\cp{2}$ has $K^2=9$. Therefore the $\QQ$-Gorenstein smoothing of $\mathbf{CP}(p_1^2,p_2^2,p_3^2)$ is $\cp{2}$ for any Markov triple $(p_1,p_2,p_3)$.

  The weighted projective space $\CC\PP(p_1^2,p_2^2,p_3^2)$ is a toric orbifold with (up to) three orbifold singularities
  \[\frac{1}{p_i^2}(1,p_iq_i-1),\qquad i=1,2,3.\]
  For each Markov number $p$, we therefore obtain a Lagrangian pinwheel $L_{p,q}\subset\cp{2}$ which is the vanishing cycle of the $\tfrac{1}{p^2}(1,pq-1)$ singularity.

  For the orbifold $\CC\PP(p_1^2,p_2^2,p_3^2)$, the numbers $q_1,q_2,q_3$ are determined by the equations
  \[q_i=\pm 3p_j/p_k\mod p_i\]
  where $i,j,k$ is a permutation of $1,2,3$. One could read this from the moment polytope for $\CC\PP(p_1^2,p_2^2,p_3^2)$, but it also follows by reducing (modulo $p_i$) the virtual dimension formula for the components of the toric (orbifold) divisor. In the literature on Markov numbers, the numbers $q_i$ are therefore related to the {\em characteristic numbers} $x_i$ of the Markov triple by $q_i=\pm 3x_i\mod p_i$ (see {\cite[II.3, Equation (5)]{Cassels}} or {\cite[p.28]{Aigner13}}). In particular, $q_i^2=-9\mod p_i$.
\end{exm}
\begin{rmk}
  The singularity $\CC^2/\Gamma_{p,q}$ is toric and its moment polytope is the cone in $\RR^2$ with edges $\RR_{\geq 0}(1,0)$ and $\RR_{\geq 0}(pq-1,p^2)$. This can be modified by a nodal trade at the singularity to give an almost toric fibration on the smoothing $B_{p,q}$: the polytope is the same but there is now a singular torus fibre at $(q,p)$ and a branch cut in the affine base connecting $(0,0)$ to $(q,p)$ (see Figure \ref{fig:almosttoric1}). The Lagrangian pinwheel $L_{p,q}$ is a ``visible surface'' in the sense of Symington {\cite[Section 7.1]{Symington}} with respect to this almost toric fibration: it is a vanishing thimble for the singular fibre, living over the vanishing path which is the branch cut. The core circle lives over the origin.
\end{rmk}

\begin{figure}[htb]
  \begin{center}
    \labellist
    \pinlabel $(q,p)$ at 120 80
    \pinlabel $(pq-1,p^2)$ at 90 250
    \endlabellist
    \includegraphics[width=100px]{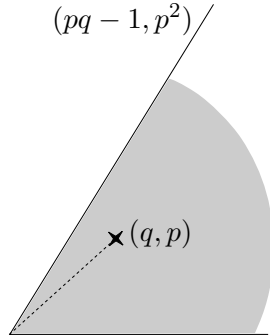}
  \end{center}
  \caption{An almost toric picture  of $B_{p,q}$ as a smoothing of a $\CC^2/\Gamma_{p,q}$. The star marks the singular fibre at $(q,p)$. The Lagrangian pinwheel $L_{p,q}$ lives over the (dashed) branch cut along $(q,p)$.}
  \label{fig:almosttoric1}
\end{figure}

\begin{rmk}\label{rmk:signamb}
  Note that $q$ is only determined up to sign modulo $p$ (which accounts for the $\pm$ signs everywhere): the rational homology balls $B_{p,q}$ and $B_{p,p-q}$ are symplectomorphic. To see this, note that the almost toric fibration for $B_{p,q}$ and for $B_{p,p-q}$ are related by an element of $GL(2,\ZZ)$ with negative determinant, which lifts to a symplectomorphism.
\end{rmk}
\begin{rmk}
  More globally, the moment polytope for the torus action on the weighted projective space $\CC\PP(p_1^2,p_2^2,p_3^2)$ is a triangle with vertices in $\ZZ^2$ whose sides have affine lengths $p_1^2,p_2^2,p_3^2$. This can be modified via nodal trades at each singularity to give an almost toric picture of $\cp{2}$; see {\cite[Corollary 2.5]{Vianna}}. The Lagrangian pinwheels $L_{p_i,q_i}$ are ``visible surfaces'' with respect to this almost toric fibration living over the branch cuts. If $(p_1,p_2,p_3)$ and $(p'_1,p_2,p_3)$ are related by a mutation then we can see both the pinwheels $L_{p_1,q_1}$ and $L_{p'_1,q'_1}$ simultaneously in the same almost toric fibration living over two halves of the same affine line (this is because mutation corresponds to the geometric operation of ``transferring the cut'' described in {\cite[Section 2]{Vianna}}). In particular, we see that $L_{p_1,q_1}$ and $L_{p'_1,q'_1}$ intersect precisely once transversely at an interior point of each pinwheel disc (see Figure \ref{fig:almosttoric2}). It follows from Theorem \ref{thm:markov} B that this intersection point cannot be removed by a Hamiltonian isotopy.
\end{rmk}

\begin{figure}[htb]
  \begin{center}
    \labellist
    \pinlabel ${L_{p_1,q_1}}$ at 44 18
    \pinlabel ${L_{p_2,q_2}}$ at 100 115
    \pinlabel ${L_{p'_1,q'_1}}$ at 80 179
    \pinlabel ${L_{p_3,q_3}}$ at 41 207
    \endlabellist
    \includegraphics{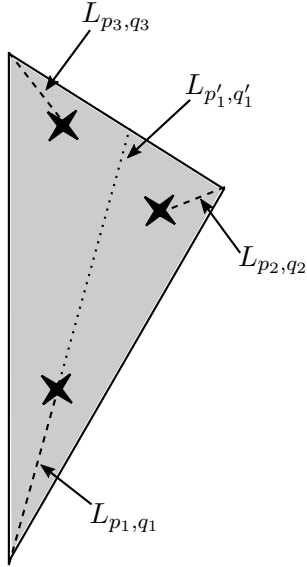}
  \end{center}
  \caption{An almost toric picture of $\cp{2}$ as a smoothing of a weighted projective space $\CC\PP(p_1^2,p_2^2,p_3^2)$; there are three singular fibres, marked by stars, with three (dashed) branch cuts. The Lagrangian pinwheels $L_{p_i,q_i}$ live over the branch cuts. The Lagrangian pinwheel $L_{p'_1,q'_1}$ is also depicted, living over the dotted line which is the continuation of one of the branch cuts. It intersects $L_{p_1,q_1}$ once transversely.}
  \label{fig:almosttoric2}
\end{figure}

\begin{dfn}
  Following Khodorovskiy \cite{KhodorovskiyQBU}, we say that a Lagrangian pinwheel $f\colon D\to X$ is {\em good} if it admits a neighbourhood $U$ such that $(U,f(D))$ is symplectomorphic to $(B_{p,q},L_{p,q})$. We call such a neighbourhood a {\em Khodorovskiy neighbourhood}. Khodorovskiy proves {\cite[Lemmas 3.3, 3.4]{KhodorovskiyQBU}} that in a neighbourhood of any Lagrangian $(p,q)$-pinwheel $L$ there is a good Lagrangian $(p,q)$-pinwheel $\mathcal{C}^0$-close to $L$ which agrees with $L$ away from a neighbourhood of its core circle.
\end{dfn}

We are interested in nondisplaceability questions about pinwheels, and start our argument by assuming we have a collection of pairwise disjoint Lagrangian pinwheels. If this is the case, then Khodorovskiy's result shows that we have a collection of pairwise disjoint good Lagrangian pinwheels and hence a collection of disjoint Khodorovskiy neighbourhoods. So, without loss of generality, we can assume that all our pinwheels are good.

\subsection{Reeb dynamics on the boundary}\label{sct:contgeom}

The boundary $\Sigma_{p,q}=\partial B_{p,q}$ is a hypersurface of contact-type diffeomorphic to the lens space $L(p^2,pq-1)$. By Gray's stability theorem, this contact structure is contactomorphic to the structure defined on $S^3/\Gamma_{p,q}\subset\CC^2/\Gamma_{p,q}$ by the contact form $\lambda=x_1dy_1+x_2dy_2$ (where $z_j=x_j+iy_j$).

The Reeb flow is given by
\[[z_1,z_2]\mapsto[e^{it}z_1,e^{it}z_2].\]
Let $g:=\gcd(p^2,pq-2)$ (where $g=p^2$ if $pq-2=0$). The Reeb orbits on $S^3$ are the fibres of the Hopf map $S^3\to S^2$. The Reeb orbits on $\Sigma_{p,q}$ give $\Sigma_{p,q}$ the structure of a Seifert-fibred space over an orbifold $S^2$.
\begin{lma}
  The generic point of the orbifold has stabiliser isomorphic to $\ZZ/(g)$.
\end{lma}
\begin{proof}
  Let $(x_0,y_0)$ be a point with $x_0\neq 0$ and $y_0\neq 0$. The Reeb orbit through this point is parametrised by $(e^{it}x_0,e^{it}y_0)$; suppose that this orbit passes through $(e^{2\pi ik/p^2}x_0,e^{2\pi ik(pq-1)/p^2}y_0)$. Then $k=k(pq-1)\mod p^2$, so $k(pq-2)=0\mod p^2$. This means that if $g=\gcd(p^2,pq-2)$ then $p^2/g$ divides $k$. The stabiliser is therefore the group of elements $\zeta^{p^2/g}$ where $\zeta\in\rt_{p^2}$. This group is isomorphic to $\rt_g$.
\end{proof}

If $p=2$, $q=1$ then $g=p^2$ and every point is generic; otherwise there are two orbifold points with stabiliser $\ZZ/(p^2)$. The singular points correspond to the Reeb orbits $z_1=0$ and $z_2=0$; we call these {\em exceptional Reeb orbits}. The generic orbit in $S^3$ is preserved by the subgroup
\[\ker(\ZZ/(p^2)\to\ZZ/(p^2/g)).\]

\begin{rmk}\label{rmk:g124}
Note that if $\gcd(p,q)=1$ then
\[g=\gcd(p^2,pq-2)=\begin{cases}
1&\mbox{ if }p=1,3\mod 4\\
2&\mbox{ if }p=0\mod 4\\
4&\mbox{ if }p=2\mod 4.
\end{cases}\]
To see this, observe that any prime divisor $\ell$ of $\gcd(p^2,pq-2)$ also divides $\gcd(p,pq-2)$, so $\ell$ divides $2$ and we see $g$ is a power of 2. If $p=2m$ then $\gcd(p^2,pq-2)=\gcd(4m^2,2(mq-1))=2\gcd(2m^2,mq-1)$. If $m$ is even then $mq-1$ is odd so $\gcd(2m^2,mq-1)=1$ (remember we are only interested in factors of 2). If $m$ is odd then $mq-1$ is even and $\gcd(2m^2,mq-1)=2$.
\end{rmk}

\subsection{Topology of rational homology balls}

The rational ball $B_{p,q}$ is homotopy equivalent to the pinwheel $L_{p,q}$ which is a Moore space $M(\ZZ/(p), 1)$. We have $H_1(\Sigma_{p,q};\ZZ)=\ZZ/(p^2)$ and $H_1(B_{p,q};\ZZ)\cong\ZZ/(p)$; the map $H_1(\Sigma_{p,q};\ZZ)\to H_1(B_{p,q};\ZZ)$ induced by inclusion is reduction modulo $p$. We may take the exceptional Reeb orbit $z_1=0$ as a generator $e$ for $H_1(\Sigma_{p,q};\ZZ)$. The generic orbit is homologous to $p^2e/g$.

Note that $H^2(B_{p,q};\ZZ)\cong\ZZ/(p)$.

\begin{lma}\label{lma:firstchernclass}
  The first Chern class $c_1(B_{p,q})\in H^2(B_{p,q};\ZZ)\cong\ZZ/(p)$ is primitive.
\end{lma}
\begin{proof}
  Since $B_{p,q}$ retracts onto $L_{p,q}$, we may restrict $TB_{p,q}$ to $L_{p,q}$. We can trivialise $TB_{p,q}$ over a neighbourhood $U$ of the core circle and over the interior $V$ of the 2-cell; the clutching function for the bundle is then given by a map $S^1\to \mathrm{Sp}(4,\RR)$, where $S^1$ is a deformation retract of $U\cap V$. The first Chern class is simply the winding number of this clutching function. This is only defined modulo $p$: we can change the trivialisation over $U$ by a map from the core circle to $\mathrm{Sp}(4,\RR)$. This changes the winding number by a multiple of $p$.

  We can compute the winding number explicitly via Khodorovskiy's local model {\cite[Theorem 3.2]{KhodorovskiyQBU}}. Let $(x_1,x_2,x_3,x_4)$ be local coordinates on a neighbourhood $S^1\times\RR^3$ of the core circle in $B_{p,q}$, with symplectic form $dx_1\wedge dx_2+dx_3\wedge dx_4$. Let $(r,\theta)$ be polar coordinates on the 2-cell in the pinwheel. The standard model for the Lagrangian immersion of the punctured 2-cell into $S^1\times\RR^3$ is
  \[f(r,\theta)=(p\theta,q(1-r)^2/2p,(1-r)\cos(q\theta),-(1-r)\sin(q\theta)).\]
  Then we have
  \begin{align*}
    f_*\partial_r&=(0,\frac{q(r-1)}{p},-\cos(q\theta),\sin(q\theta))\\
    f_*\partial_\theta&=(p,0,-q(1-r)\sin(q\theta),-q(1-r)\cos(q\theta)).
  \end{align*}
  Let $m=|f_*\partial_\theta|=p|f_*\partial_r|=\sqrt{p^2+q^2(1-r)^2}$, let $J$ be the standard complex structure $J\partial_{x_1}=\partial_{x_2}$, $J\partial_{x_3}=\partial_{x_4}$, and let $M_{r,\theta}$ be the matrix whose columns are
  \[\frac{p}{m}f_*\partial_r,\ \frac{p}{m}Jf_*\partial_r,\ \frac{1}{m}f_*\partial_\theta,\ \frac{1}{m}Jf_*\partial_\theta,\]
  that is
  \[M_{r,\theta}=\left(\begin{array}{cc}
                         \frac{iq(1-r)}{m}& \frac{p}{m}\\
                         -\frac{p}{m}e^{-iq\theta}& -\frac{q(1-r)}{m}ie^{-iq\theta}
                       \end{array}\right).\]
  This is a unitary matrix whose inverse sends $(1,0,0,0)$ to the normalised vector in the $f_*\partial_r$-direction and $(0,0,1,0)$ to the normalised vector in the $f_*\partial_\theta$-direction. Therefore $M^{-1}_{1,\theta}$ is the clutching function for the bundle $TB_{p,q}$ described in the first paragraph. This has determinant $e^{iq\theta}$, which has winding number $q$ as $\theta$ runs around the circle. Since $\mathrm{Sp}(4,\RR)\simeq U(2)$ and $\det_*\colon\pi_1(U(2))\to \pi_1(U(1))$ is an isomorphism, this tells us that the first Chern class of $TB_{p,q}$ is $q\in\ZZ/(p)$, which is primitive modulo $p$.     
\end{proof}

\subsection{Extrinsic topology of Lagrangian pinwheels}

Suppose that $(X,\omega)$ is a symplectic 4-manifold with $H_1(X;\ZZ)=0$, $H_2(X;\ZZ)=\ZZ$, and $B_i\subset X$, $i=1,\ldots,N$, is a collection of pairwise disjointly embedded symplectic rational homology balls $B_i\cong B_{p_i,q_i}$. Let $\Sigma_i$ denote $\partial B_i$, let $B=\coprod_{i=1}^NB_i$, $\Sigma=\coprod_{i=1}^N\Sigma_i$ and $V=X\setminus B$. Write $\jmath\colon V\to X$ for the inclusion map. Let $\Delta=\prod_{i=1}^Np_i$ and let $L_i$ denote the Lagrangian pinwheel in $B_i$.

\begin{lma}
The numbers $p_i$ are pairwise coprime.
\end{lma}
\begin{proof}
  Suppose that $d$ is a prime divisor of $p_i$. The pinwheel $L_i$ defines a nontrivial class in $H_2(B_i;\ZZ/(d))$ which satisfies $[L_i]^2\neq 0\mod d$. If $d\neq 1$ were a common prime divisor of $p_i$ and $p_j$ then $L_i$ and $L_j$ would define nontrivial classes in $H_2(X;\ZZ/(d))$. Since $H_2(X;\ZZ/(d))=\ZZ/(d)$ and the intersection pairing is nontrivial, this implies that $[L_i]\cdot [L_j]\neq 0$, contradicting the fact that $L_i$ and $L_j$ are disjoint.
\end{proof}

\begin{lma}
  If $c_1(X)$ is divisible by $d$ in $H^2(X;\ZZ)$ and $B_{p,q}$ is symplectically embedded in $X$ then $\gcd(p,d)=1$.
\end{lma}
\begin{proof}
  Let $\iota\colon B_{p,q}\to X$ be the inclusion map. Since this is a symplectic map, $\iota^*c_1(X)=c_1(B_{p,q})\in H^2(B_{p,q};\ZZ)\cong\ZZ/(p)$. This implies $c_1(B_{p,q})$ is a multiple of $d$ in $\ZZ/(p)$, so has order at most $p/\gcd(p,d)$. Because $c_1(B_{p,q})$ is primitive, we deduce that $\gcd(p,d)=1$.
\end{proof}

\begin{lma}\label{lma:homology}
  We have:
  \begin{enumerate}
  \item[(a)] $H_1(V;\ZZ)=0$,
  \item[(b)] $\jmath_*\colon H_2(V;\ZZ)\to H_2(X;\ZZ)=\ZZ$ is the inclusion of the ideal $(\Delta)$.
  \item[(c)] $H^2(V;\ZZ)\cong H_2(V,\Sigma;\ZZ)\cong\ZZ\oplus T$ where
    \[T=\ZZ/(\Delta/m)\]
    for some divisor $m$ of $\Delta$.
  \item[(d)] We define $\eE\in H_2(V,\Sigma;\QQ)$ to be the unique element such that $m\eE$ is a generator for the lattice $H_2(V,\Sigma;\ZZ)/T\subset H_2(V,\Sigma;\QQ)$. The map
    \[\jmath^*\colon H^2(X;\QQ)\to H^2(V;\QQ),\]
    (or, Poincar\'e-dually, $\jmath_!\colon H_2(X;\QQ)\to H_2(V,\Sigma;\QQ)$) sends a generator $H\in H^2(X;\ZZ)$ to $\Delta\eE\in H_2(V,\Sigma;\QQ)$.
  \end{enumerate}
\end{lma}
\begin{proof}
The Mayer-Vietoris sequence gives
\begin{gather*}0\to H_2(B;\ZZ)\oplus H_2(V;\ZZ)\to H_2(X;\ZZ)\to\\
  \qquad \to H_1(\Sigma;\ZZ)\to H_1(B;\ZZ)\oplus H_1(V;\ZZ)\to 0\end{gather*}
We have $H_2(B;\ZZ)=0$ and $H_1(\Sigma;\ZZ)=\TO{\ZZ/(\Delta^2)}{\mod\Delta}{\ZZ/(\Delta)}=H_1(B;\ZZ)$.

(a) Suppose that $H_1(V;\ZZ)\neq 0$. The existence of a surjective map of abelian groups $\ZZ/(\Delta^2)\to \ZZ/(\Delta)\oplus H_1(V;\ZZ)$ tells us that $\ZZ/(\Delta)\oplus H_1(V;\ZZ)$ is cyclic. Therefore, by the classification of abelian groups, $H_1(V;\ZZ)$ is cyclic of order coprime to $\Delta$. But there is no surjective map $\ZZ/(\Delta^2)\to\ZZ/(m)$ unless $m$ divides $\Delta^2$.

(b) Part (a) implies that the Mayer-Vietoris sequence breaks up as
\[0\to H_2(V;\ZZ)\to H_2(X;\ZZ)=\ZZ\to\ZZ/(\Delta)\to 0,\]
which implies (b).

(c) Poincar\'e-Lefschetz duality tells us that $H^2(V;\ZZ)\cong H_2(V,\Sigma;\ZZ)$. The space $X/V$ is the wedge $\bigvee_{i=1}^N(B_i/\Sigma_i)$, and Lefschetz duality further implies that $H^i(B_i / \Sigma_i) \cong H_{4-i}(B_i)$ for $i>0$, so
\[H^2(B_i/\Sigma_i;\ZZ)=0,\quad H^3(B_i/\Sigma_i;\ZZ)=\ZZ/(p_i).\]
The long exact sequence in relative cohomology:
\[H^2(X/V;\ZZ)\to H^2(X;\ZZ)\to H^2(V;\ZZ)\to H^3(X/V;\ZZ)\to H^3(X;\ZZ)\]
becomes
\[0\to \ZZ\stackrel{\jmath^*}{\to} H^2(V;\ZZ)\to \ZZ/(\Delta)\to 0.\]
Therefore, $H_2(V,\Sigma;\ZZ)\cong\ZZ\oplus T$ where $T$ is torsion, and $T$ embeds as a subgroup of $\ZZ/(\Delta)$, so $T\cong \ZZ/(m)$ for some $m$ dividing $\Delta$ and $\jmath^*H=(\Delta/m,\ell)$ for some $\ell\in T$.

(d) From part (c), over $\QQ$, we have $\jmath^*H=(\Delta/m)m\eE=\Delta\eE$.
\end{proof}

\section{Orbifolds}\label{sct:orbi}

\subsection{Almost complex structures}

  Let $(X,\omega)$ be a symplectic manifold and $B_i\subset X$, $i=1,\ldots,N$, be a collection of pairwise disjoint symplectic embeddings of rational homology balls $B_i\cong B_{p_i,q_i}$. Let $\Sigma_i$ denote the boundary $\partial B_i$ and let $\Sigma=\bigcup_{i=1}^N\Sigma_i$. Let $V=X\setminus \coprod_{i=1}^NB_i$.

A neighbourhood $U_i$ of $\Sigma_i\subset X$ is symplectomorphic to a neighbourhood of $\orb{\Sigma}_i=S^3/\Gamma_{p_i,q_i}\subset\CC^2/\Gamma_{p_i,q_i}$. For each $i$, we fix a symplectic embedding $\phi_i\colon U_i\to \CC^2/\Gamma_{p_i,q_i}$ sending $\Sigma_i$ to $\orb{\Sigma}_i$. Let $\orb{B}_i$ denote the compact component of $\left(\CC^2/\Gamma_{p_i,q_i}\right)\setminus\orb{\Sigma}_i$. Define the symplectic orbifold $(\orb{X},\orb{\omega})$ by
\[\orb{X}=V\cup_{\phi_i(\Sigma_i)\cong\orb{\Sigma}_i}\coprod_{i=1}^N\orb{B}_i.\]
This orbifold has a singularity $\orb{x}_i\in\orb{B}_i$ of type $\frac{1}{p_i^2}(1,p_iq_i-1)$ for each $i=1,\ldots,N$.

In what follows, let $(\omega_i,I_i)$ denote the standard K\"ahler structure on $\CC^2/\Gamma_{p_i,q_i}$. Let $\mathcal{J}_{\orb{X}}$ denote the space of compatible almost complex structures on $\orb{X}$ which agree with $I_i$ on $\orb{B}_i$ and let $\mathcal{J}_\Sigma$ denote the space of compatible almost complex structures on $X$ which agree with $\phi_i^*I_i$ on $U_i$.

The following lemmas are proved by Vianna:

\begin{lma}[{\cite[Example 3.3]{Vianna}}]\label{lma:adjusted}
  Any $J\in\mathcal{J}_\Sigma$ is an {\em adjusted} almost complex structure in the sense of {\cite[Section 2.2]{BEHWZ}}. In particular, the almost complex structure $J|_V$ extends in a canonical way to an almost complex structure $\overline{J|_V}$ on the symplectic completion $(\bar{V},\overline{\omega|_V})$.
\end{lma}

\begin{lma}[{\cite[Claim 3.1]{Vianna}}]\label{lma:endkaehler}
  The noncompact end of $\left(\overline{V},\overline{J|_V}\right)$ is isomorphic as a K\"ahler manifold to a neighbourhood of the negative end of $\coprod_{i=1}^N\left(\CC^2/\Gamma_{p_i,q_i}\right)\setminus\{(0,0)\}$.
\end{lma}

This implies:

\begin{lma}\label{lma:jhat}
  Given an almost complex structure $J\in\mathcal{J}_\Sigma$ there is a unique almost complex structure $\orb{J}\in\mathcal{J}_{\orb{X}}$ such that $\left(\orb{X}\setminus\{\orb{x}_i\}_{i=1}^N,\orb{\omega},\orb{J}\right)$ is isomorphic to $\left(\overline{V},\overline{\omega|_V},\overline{J|_V}\right)$.
\end{lma}

The following lemma is also implicit in \cite{Vianna}. See {\cite[Definition 2.1.3]{ChenRuan}} for the definition of an orbifold holomorphic curve.

\begin{lma}\label{lma:orbifoldcompactification}
  Let $(X,\omega,J)$ and $(\orb{X},\orb{\omega},\orb{J})$ be as above. Let $S$ be a punctured Riemann surface and $u\colon S\to\bar{V}$ be a proper finite-energy punctured $\overline{J|_V}$-holomorphic curve. There is a compact orbifold Riemann surface $\orb{S}$ and an orbifold $\orb{J}$-holomorphic map $\orb{u}\colon\orb{S}\to\orb{X}$ which extends $u$ (where we identify $\bar{V}$ with $\orb{X}\setminus\{\orb{x}_i\}_{i=1}^N$). Conversely, suppose that $\orb{u}\colon\orb{S}\to\orb{X}$ is an orbifold $\orb{J}$-holomorphic map from an orbifold Riemann surface $\orb{S}$ to $\orb{X}$ and let $S=\orb{u}^{-1}\left(\bar{V}\right)$. Then $u:=\orb{u}|_S\colon S\to\bar{V}$ is a finite-energy punctured holomorphic curve.
\end{lma}
\begin{proof}
  Consider a neighbourhood $N$ of a component of the negative end of $\overline{V}$, isomorphic to a neighbourhood of the negative end of $\left(\CC^2/\Gamma_{p_i,q_i}\right)\setminus\{0\}$. Let $C=u^{-1}(N)\subset S$. The map $u$ is asymptotic to a collection of Reeb cylinders, so, for sufficiently small $N$, $C=\bigcup_jC_j$ is a collection of annuli $C_j\cong D^2\setminus\{0\}$ in $S$. If we pass to the uniformising cover $\CC^2\to\CC^2/\Gamma_{p_i,q_i}$ then the map $u|_{C_j}$ lifts to a map $\tilde{u}_j\colon\tilde{C}_j\to\CC^2\setminus\{0\}$ defined on a finite covering space $\tilde{C}_j$ which is also an annulus. Let $\gamma_j$ denote the action of the deck group of this cover and let $\tilde{D}_j=\tilde{C}_j\cup\{0\}$ be a disc compactifying $\tilde{C}_j$; the deck group $\gamma_j$ continues to act on $\tilde{D}_j$ fixing the origin. The map $\tilde{u}_j$ is also asymptotic to a Reeb cylinder and hence extends continuously to a map $\orb{u}_j\colon D_j\to\CC^2$. By the removal of singularities theorem, $\orb{u}_j$ is holomorphic. Therefore we compactify $S$ by attaching the (orbifold) discs $D_j=\tilde{D}_j/\gamma_j$ along the annuli $C_j$. We use the lifts $\orb{u}_j$ to define an orbifold extension of $u$ over the orbifold discs $D_j$.

  For the converse, since $\orb{S}$ is compact, the curve $\orb{u}$ has finite symplectic area; hence $u$ has finite area and is a finite energy curve.
\end{proof}

\subsection{Homology classes of orbifold curves}\label{sct:homology}

A punctured finite energy curve $u\colon S\to \bar{V}$ defines a relative homology class $[u]\in H_2(V,\Sigma;\ZZ)$, and we can identify this relative homology group with $H_2(\orb{X};\ZZ)$ as $\orb{X}$ is homeomorphic to the quotient $V/\Sigma$. For a 4-dimensional orbifold $\orb{X}$, W.~Chen \cite{ChenOrbiAdj} defines an intersection pairing on the rational homology $H_2(\orb{X};\QQ)$ which coincides with the usual intersection pairing on the subset $V\subset\orb{X}$. In particular, this means that if $C\in H_2(X;\QQ)$ then $C^2=(\jmath_!C)^2$, where $\jmath_!\colon H_2(X;\QQ)\to H_2(V,\Sigma;\QQ)$ is induced by the inclusion map $\jmath\colon V\to X$.

\begin{rmk}
  An orbifold curve $C$ defines a class in $H_2(V,\Sigma;\ZZ)/T$, which can therefore be written as $[C]=D\eE$ for some $D\in m\ZZ$, where $m\in\ZZ$ is the integer from Lemma \ref{lma:homology}.
\end{rmk}

\begin{lma}
  Let $\eE\in H_2(V,\Sigma;\QQ)$ be the element introduced in Lemma \ref{lma:homology}. With respect to the intersection pairing, $\eE^2=1/\Delta^2$.
\end{lma}
\begin{proof}
  By Poincar\'e duality over $\ZZ$, the intersection pairing on $H_2(X;\ZZ)$ is unimodular, so if $H$ is a generator of $H_2(X;\ZZ)$, we have $H^2=1$. By Lemma \ref{lma:homology}, $\jmath_!H=\Delta\eE$, so we deduce that $\eE^2=1/\Delta^2$.
\end{proof}

\begin{lma}
  Suppose that $c_1(X)=kH$. The first Chern class $c_1(\orb{X})$ can be identified via pullback and Poincar\'e duality with the class $k\Delta\eE\in H_2(V,\Sigma;\QQ)$. Therefore, if we have an orbifold curve with homology class $C=D\eE$ then $c_1(\orb{X})\cdot C=\frac{kD}{\Delta}$.
\end{lma}
\begin{proof}
  By composing the pullback map along $V\hookrightarrow \orb{X}$ with the Poincar\'e-Lefschetz duality isomorphism, we get an isomorphism
  \[H^2(\orb{X};\QQ)\to H^2(V;\QQ)\to H_2(V,\Sigma;\QQ).\]
  Since the first Chern class is natural under pullback, $c_1(\orb{X})$ pulls back to $c_1(V)\in H^2(V;\QQ)$ which is also the pullback of $c_1(X)$ to $V$. Since $c_1(X)=kH$ and $H$ pulls back to $\Delta\eE\in H^2(V,\Sigma;\QQ)$ we see that $c_1(\orb{X})$ is identified with $k\Delta\eE\in H_2(V,\Sigma;\QQ)$.
\end{proof}

\subsection{Orbifold adjunction}

W. Chen has proved the following adjunction formula for orbifold holomorphic curves. Here we state it only in the case where the orbifold $\orb{X}$ has isolated singularities $\orb{X}_{\text{sing}}$ so that a generic point on the orbifold curve has isotropy of order one.

We introduce some notation. Let $\orb{S}$ be an orbifold Riemann surface, $f\colon \orb{S} \to \orb{X}$ a somewhere-injective orbifold holomorphic map, let $H_z$ be the isotropy group of $\orb{S}$ at $z\in\orb{S}$ and let $G_z$ denote the isotropy group of $\orb{X}$ at $f(z)$. Let $Z\subset \orb{S}$ be the collection of orbifold points, that is the set of points with $|H_z|>1$. Let $D$ be a disc neighbourhood of $z$, let $B\subset\CC^2$ be a ball and let $B\to\orb{X}$ be a local uniformising cover of a neighbourhood of the orbifold point $f(z)$. The orbifold holomorphic curve is defined by an injective homomorphism $\rho_z\colon H_z\to G_z$, a branched cover $\tilde{D}\to D$ with deck group $H_z$ and the collection of all lifts $\{\tilde{f}_\alpha\colon\tilde{D}\to B\}_{\alpha\in A}$ of $f$ to the local uniformising cover of $\orb{X}$ at $f(z)$, equivariant with respect to the actions of $H_z$ and $G_z$ and the homomorphism $\rho_z$:
\begin{equation}\label{eq:equivariance}\tilde{f}_\alpha(\zeta x)=\rho_z(\zeta)\cdot \tilde{f}_\alpha(x),\quad x\in \tilde{D}.\end{equation}
Here $A$ is just a set indexing all the lifts; since $f$ is somewhere-injective, there are $|G_z/H_z|$ of these lifts and between any two lifts (not necessarily distinct) there is a well-defined local intersection number $\tilde{f}_\alpha\cdot\tilde{f}_\beta$.  Define
\begin{equation}\label{eq:kz}k_z:=\frac{1}{|G_z|}\left(\sum_\alpha\tilde{f}_\alpha\cdot\tilde{f}_\alpha+\frac{1}{2}\sum_{\alpha\neq\beta}\tilde{f}_\alpha\cdot\tilde{f}_\beta\right).\end{equation}
  For distinct $z \neq z'\in\orb{S}$, if we denote the unordered pair as $[z,z']$, the number $k_{[z,z']}$ is defined using lifts $\tilde{f}_\alpha$ at $z$ and $\tilde{f}'_\beta$ at $z'$ to be equal to
  \begin{equation}\label{eq:kzz}k_{[z,z']}:=\frac{1}{|G_z|}\sum_{\alpha,\beta}\tilde{f}_\alpha\cdot\tilde{f}'_\beta.\end{equation}

\begin{thm}[{\cite[Theorem 3.1]{ChenOrbiAdj}}]
  Let $f\colon\orb{S}\to \orb{X}$ be a somewhere-injective orbifold holomorphic curve representing a homology class $C\in H_2(\hat{X};\QQ)$. Let $g_{|\orb{S}|}$ denote the genus of the underlying smooth Riemann surface. Then
  \begin{equation}\label{eq:adj}1+\frac{C\cdot C-c_1(\orb{X})\cdot C}{2}=g_{|\orb{S}|}+\frac{1}{2}\sum_{z\in Z}\left(1-\frac{1}{|H_z|}\right)+\sum_{z\in \orb{S}}k_z+\sum_{\mathclap{\substack{\orb{S}\ni z\neq z'\in \orb{S}\\f(z)=f(z')}}}k_{[z,z']}.\end{equation}
\end{thm}

\begin{rmk}
  The local intersection number of two curves at a point $y$ is nonnegative. It is one if and only if the two curves are embedded at $y$ and intersect one another transversely. It is zero if and only if the two curves are identical and embedded at $z$; in the sum for $k_{[z,z']}$, the two branches of the curve passing through $f(z)=f(z')$ with $z\neq z'$ are considered to be different even if their geometric image coincides, so $k_{[z,z']}$ is always positive if there are distinct $z$ and $z'$ in $\orb{S}$ mapping to the same point in $\orb{X}$. By contrast, it is possible for $k_z$ to be zero if $G_z=H_z$ and the single lift $\tilde{f}$ is embedded at $z$, in other words if the curve is a {\em suborbifold} at $z$.
\end{rmk}

\begin{rmk}\label{rmk:simplify}
  In our case, the group $G_z$ is a cyclic group. If we pick $\zeta\in G_z$ and write $\alpha\mapsto\zeta\alpha$ for its action on the indices of the branches $\tilde{f}_\alpha$ then we see that
  \[\tilde{f}_\alpha\cdot\tilde{f}_\beta=\tilde{f}_{\zeta\alpha}\cdot\tilde{f}_{\zeta\beta}\]
  so that many terms in the sum defining $k_z$ are repeated. Therefore,
  \[k_z=\frac{1}{2|H_z|}\left(2N_0+\sum_{i=1}^{\mathclap{|G_z/H_z|-1}} N_i\right),\]
  where $N_i=\tilde{f}_\alpha\cdot\tilde{f}_{\zeta^i\alpha}$. We define
  \[K_z:=2N_0+\sum_{i=1}^{\mathclap{|G_z/H_z|-1}}N_i,\]
  so that $k_z=K_z/2|H_z|$.
\end{rmk}

\begin{rmk}\label{rmk:locmod}
Recall that we assume a neighbourhood of each singular point $f(z)\in \orb{X}_{\mathrm{sing}}$ is biholomorphic to a neighbourhood of zero in the standard model $\CC^2/G_z$. In this case, any component of the lift $\tilde{f}_{\alpha}(z)$ can be written as $\left(\sum_{i=0}^\infty a_iz^{Q_i},\sum_{i=1}^\infty b_iz^{R_i}\right)$ for some power series $\sum a_iz^{Q_i}$ and $\sum b_iz^{R_i}$ with $a_0,b_1\neq 0$, convergent on a neighbourhood of $0$. Suppose that $Q_0<R_1$; equivalently, we assume that $\CC\times\{0\}\subset\CC^2$ is the tangent plane of the branched minimal immersion $\tilde{f}_{\alpha}$ at the origin. Since $a_0\neq 0$, we may take a branch $F(z)$ of $\sqrt[\leftroot{-4}\uproot{4}Q]{\sum_{i=0}^\infty a_iz^{Q_i-Q}}$ around $z=0$ and use the local coordinate $w=zF(z)$ on the domain; in these coordinates, the orbifold curve is $f(w)=(w^{Q},c_1w^{R_1}+c_2w^{R_2}+\cdots)$ (where we now set $Q=Q_0$). This is a simple version of the local model for branched minimal immersions established by Micallef and White \cite{MW}.
\end{rmk}

\begin{rmk}\label{rmk:equiv}
  The equivariance condition \eqref{eq:equivariance} gives some constraints on the exponents $Q$, $R_1$, $R_2,\ldots$. Given a $|H_z|$th root of unity $\zeta$, suppose that the action of $\rho_z(\zeta)$ on $(z_1,z_2)$ is $(\zeta^{m_1}z_1,\zeta^{m_2}z_2)$ for some $0\leq m_1,m_2<|H_z|$. Then Equation \eqref{eq:equivariance} becomes
\[\left(\zeta^Qz^Q,\sum b_i\zeta^{R_i}z^{R_i}\right)=\left(\zeta^{m_1}z^Q,\sum b_i\zeta^{m_2}z^{R_i}\right)\]
which implies $Q\equiv m_1\mod |H_z|$ and $R_i\equiv m_2\mod |H_z|$.
\end{rmk}

We can express the adjunction contributions $k_z$ in terms of the exponents $Q$, $R_1$, $R_2,\ldots$ as follows.

\begin{lma}\label{lma:jet}
  \begin{enumerate}
  \item Let $u(z)=\left(z^Q,h(z)\right)$ be a germ of a somewhere-injective holomorphic curve in $\CC^2$ with $h(z)=\sum_{i=1}^\infty a_iz^{R_i}$. Let $K$ denote the local adjunction contribution from the singularity at $(0,0)$. Then
    \[K-1\geq QR_1-Q-R_1\]
    with equality if and only if $\gcd(Q,R_1)=1$. Suppose, moreover, that $R_i\equiv R_1\mod d$ for some $d$ and for all $i$. Then $K-1\neq QR_1-Q-R_1$ implies $K-1>d$.
  \item Let $u_i(z)=\left(z^Q,h_i(z)\right)$, $i=0,1$, be germs of somewhere-injective holomorphic curves in $\CC^2$ with $h_i(z)=\sum_{j=1}^\infty a_{i,j}z^{R_j}$. Then the local intersection contribution $\nu$ to $u_0\cdot u_1$ from the intersection point $(0,0)$ satisfies $\nu\geq QR_1$. Suppose, moreover, that $R_j\equiv R_1\mod d$ for some $d$ and for all $j$. Then strict inequality $\nu\neq QR_1$ implies $\nu>d$.
  \end{enumerate}
\end{lma}
\begin{proof}
  \begin{enumerate}
  \item Define $g_j=\gcd(Q,R_1,R_2,\ldots,R_{j-1})$. Since the orbifold map is some\-where-injective it is not multiply-covered, and hence there exists $M>0$ such that $g_m=1$ for all $m>M$. Milnor {\cite[Remark 10.10]{Milnor}} (see also {\cite[Remark after Theorem 7.3]{MW}}) gives a formula for the local adjunction contribution in terms of $g_j$:
    \begin{gather*}K=(Q-g_2)(R_1-1)+(g_2-g_3)(R_2-1)+\cdots\\
      \cdots+(g_{M}-1)(R_M-1)\end{gather*}
    We see that $K-1\geq QR_1-Q-R_1$ with equality if and only if $M=1$, that is, if and only if $\gcd(Q,R_1)=1$.

    Suppose that $K-1>QR_1-Q-R_1$. Then, since $\gcd(Q,R)\neq 1$ we must have $Q\geq 2$. Moreover, since $R_1>Q$ we must have $R_1>2$. Since $R_k=R_1\mod d$ and $R_k>R_1$ we also have $R_k>d+2$ for $k\geq 2$. Let $k\geq 2$ be minimal such that $g_k>g_{k+1}$ (which exists because $g_2\neq 1$ and $g_{M+1}=1$).

    We have
    \begin{align*}
      K-1&\geq (g_k-g_{k+1})(R_k-1)-1\\
      &\geq R_k-2\\
      &>d\mbox{ since }k\geq 2.
    \end{align*}
  
\item The local intersection contribution $\nu$ is given by Equation 2 in {\cite[Theorem 7.1]{MW}}:
  \[\nu=\sum_{\zeta\in\rt_Q}\mbox{order of vanishing of }h_0(\zeta z)-h_1(z).\]
  For each $\zeta\in\rt_Q$, the first potentially nonzero term in $h_0(\zeta z)-h_1(z)$ is $\left(a_{1,1}\zeta^{R_1}-a_{2,1}\right)z^{R_1}$ which means $h_0(\zeta z)-h_1(z)$ has order at least $R_1$. Therefore $\nu\geq QR_1$. Inequality means that, for some $\zeta$, the highest order term in $h_0(\zeta z)-h_1(z)$ is $\left(a_{1,j}\zeta^{R_j}-a_{2,j}\right)z^{R_j}$, which has order $R_j>d$. This gives $\nu>d$ as required.
  \end{enumerate}
\end{proof}
  
\begin{cor}\label{cor:gcd}
  Suppose we have an orbifold $\orb{X}$ with a singularity at $\orb{x}$ of type $\frac{1}{p^2}(pq-1,1)$ and an orbifold holomorphic curve $f\colon\orb{S}\to\orb{X}$ with $f(z)=\orb{x}$ for some point $z\in\orb{S}$ with $|H_z|=d_z$. We have
  \[K_z-1\geq\frac{p^2}{d_z}QR_1-Q-R_1\]
  with equality only if $\gcd(Q,R_1)=1$; strict inequality implies $K_z-1>d_z$.
\end{cor}
\begin{proof}
  We have a total of $p^2/d_z$ lifts, $\tilde{f}_\alpha$, of $f$ to the local uniformising cover, which are locally given by $\tilde{f}_\alpha(z)=\left(z^Q,\sum a_{\alpha,i}z^{R_i}\right)$. Recall that
  \[K_z=2N_0+N_1+\cdots+N_{(p^2/d_z)-1}\]
  where $2N_0$ is the local adjunction contribution for each $\tilde{f}_\alpha$ (independent of $\alpha$ since the lifts are conjugated by the residual $G_z/H_z$-action) and $N_i$ ($i\neq 0$) is the local intersection contribution $\tilde{f}_\alpha\cdot\tilde{f}_{\zeta^i\cdot\alpha}$. From Lemma \ref{lma:jet}, we get
  \begin{align*}
    2N_0&\geq QR_1-Q-R_1+1,\\
    N_i&\geq QR_1,\quad i\neq 0,
  \end{align*}
  so
  \[K_z-1\geq\frac{p^2}{d_z}QR_1-Q-R_1.\]
  Equality implies that $2N_0=QR_1-Q-R_1+1$, which implies $\gcd(Q,R_1)=1$. Strict inequality implies that either $2N_0-1>QR_1-Q-R_1$ or $N_i>QR_1$. In the first case, Lemma \ref{lma:jet}(1) implies that $2N_0-1>d_z$; in the second case, Lemma \ref{lma:jet}(2) implies that $N_i>d_z$. In either case, we see that $K_z-1>d_z$.
\end{proof}

\subsection{Virtual dimension}

  W. Chen {\cite[Section 1]{ChenPseudo}} gives the following formula for the virtual complex dimension of the moduli space of orbifold holomorphic maps $f\colon\orb{S}\to\orb{X}$ in the class $[\orb{C}]$ where $\orb{S}$ is an orbifold Riemann surface with underlying smooth surface of genus zero and with orbifold points $Z\subset\orb{S}$:
  \begin{equation}\label{eq:vdim}d=c_1(\orb{X})\cdot[\orb{C}]+2-(3-|Z|)-\sum_{z\in Z}\frac{m_{1,z}+m_{2,z}}{|H_z|}\end{equation}
  where $0<m_{1,z},m_{2,z}<|H_z|$ are integers related to the homomorphism $\rho_z\colon H_z\to G_{z}$ (where $G_{z}$ is the isotropy group at $f(z)$). Specifically, in local coordinates on the uniformising cover $\CC^2\to\CC^2/G_z$ at $f(z)$,
  \[\rho_z(\zeta)(z_1,z_2)=(\zeta^{m_{1,z}}z_1,\zeta^{m_{2,z}}z_2).\]
  If, with respect to these coordinates, a lift of $f$ to the uniformising cover is given by $\tilde{f}_\alpha(z)=\left(z^Q,\sum a_iz^{R_i}\right)$ then, by Remark \ref{rmk:equiv}, we have $Q=m_{1,z}\mod |H_z|$ and $R_i=m_{2,z}\mod |H_z|$.
  
  \begin{rmk}\label{rmk:vdim}
    Recall from Section \ref{sct:contgeom} that, unless $p=2$, there are two exceptional Reeb orbits in $\Sigma_{p,q}$; these are the intersections of $\Sigma_{p,q}$ with the coordinate lines $\CC\times\{0\}$ and $\{0\}\times\CC$. In the case when $\Gamma=\Gamma_{p,q}$, if the punctured curve $C$ is asymptotic to one of the exceptional Reeb orbits in $\Sigma_{p,q}$ then each lift $\tilde{f}_{\alpha}$ of the orbifold curve $\orb{C}$ is tangent to the corresponding complex plane $\CC\times\{0\}$ or $\{0\}\times\CC$ in $\CC^2$. Since the action of $\Gamma_{p,q}$ on $\CC^2$ is
  \[\rho(\zeta)(z_1,z_2)=(\zeta z_1,\zeta^{pq-1}z_2),\]
  we know from Remark \ref{rmk:equiv} that either $Q=m_{1,z}=1\mod |H_z|$ and $R_1=m_{2,z}=pq-1\mod |H_z|$ or else $Q=m_{2,z}=pq-1\mod |H_z|$ and $R_1=m_{1,z}=1\mod |H_z|$. We also note that
  \[(pq-1)^{-1}=-pq-1\mod p^2\]
  and $|H_z|$ divides $p^2$. Thus $R_1=Q(\pm pq-1)\mod |H_z|$; we see that this sign ambiguity is inherited from the sign ambiguity mentioned in Remark \ref{rmk:signamb}.

  Suppose that $z$ is the only point in $\orb{C}$ which maps to the orbifold point $\orb{x}$ and that the homology class $[\orb{C}]$ maps to a nonzero class in $H_1(\Sigma_{p,q};\ZZ)$ under the connecting homomorphism. Then the corresponding punctured curve $C$ is asymptotic to a multiple of one of the exceptional Reeb orbits, as the generic Reeb orbits are homologically trivial in $\Sigma_{p,q}$. This will be the case whenever we need to make use of this remark.
  \end{rmk}
  
\section{Curves in the orbifold}\label{sct:orbicurve}

\subsection{Standing notation}

We first establish some standing notation for this section.

Let $X=\cp{2}$, let $N$ be a positive integer and let $B_i\subset X$, $i=1,\ldots,N$, be collection of pairwise disjoint symplectic embeddings of rational homology balls $B_i\cong B_{p_i,q_i}$ where
\[p_1<p_2<\cdots<p_N.\]
Define $\Delta:=\prod_{i=1}^Np_i$. Let $\Sigma_i$ denote the boundary $\partial B_i$, let $\Sigma=\bigcup_{i=1}^N\Sigma_i$, $B=\bigcup_{i=1}^NB_i$, and let $V=X\setminus B$. Let $\orb{X}_{\mathrm{sing}}=\{\orb{x}_i\}_{i=1}^N$ denote the set of orbifold points of $\orb{X}$ and let $\orb{J}\in\mathcal{J}_{\orb{X}}$. Recall that we have a class $\eE\in H_2(V,\Sigma;\QQ)$ such that $\eE^2=1/\Delta^2$ such that any surface in $H_2(V,\Sigma;\ZZ)$ is a multiple of $\eE$.

Let $\orb{C}$ be a somewhere-injective $\orb{J}$-holomorphic orbifold curve in $\orb{X}$. Let $f\colon\orb{S}\to\orb{X}$ be a $\orb{J}$-holomorphic parametrisation of $\orb{C}$ where $\orb{S}$ is some orbifold Riemann surface whose underlying topological surface has genus zero. Let $Z\subset \orb{S}$ denote the set of orbifold points on $\orb{S}$.

\subsection{Low degree curves}\label{sct:lowdeg}

Recall (Lemma \ref{lma:homology}(d)) that under the natural map
\[\jmath_!\colon H_2(X;\QQ)\to H_2(V,\Sigma;\QQ)\cong H_2(\orb{X};\QQ),\]
the class of a complex line is sent to $\Delta\eE$. Let $[\orb{C}]=D\eE\in H_2(\orb{X};\QQ)$ be the homology class of $\orb{C}$. We will now assume that $D\leq\Delta$ and see what restrictions this imposes on $\orb{C}$.

The adjunction formula, together with Remark \ref{rmk:simplify} and the fact that the underlying curve $\orb{S}$ has genus zero, says that:
\begin{equation}\label{eq:adj2}
  1=\frac{3\Delta D-D^2}{2\Delta^2}+\sum_{z\in Z}\frac{1}{2}\left(1-\frac{1}{d_{z}}\right)+\sum_{z\in \orb{S}} \frac{K_{z}}{2d_{z}}+\sum_{\mathclap{\substack{\orb{S}\ni z\neq z'\in \orb{S}\\f(z)=f(z')}}}k_{[z,z']},
\end{equation}
where $d_{z}$ is the size of the isotropy group $H_{z}$ for the curve $\orb{C}$ at $z$.

\begin{lma}\label{lma:kzzz}
  Suppose that $D\leq\Delta$. If $f(z)=f(z')$ then $z,z'\in Z$. If $k_z\neq 0$ then $z\in Z$. In other words, the only $k_z$ and $k_{[z,z']}$ contributions can come from $z,z'\in Z$ and the sums $\sum_{z\in \orb{S}}k_z$ and $\sum_{\substack{\orb{S}\ni z\neq z'\in\orb{S}\\f(z)=f(z')}} k_{[z,z']}$ in Equation \eqref{eq:adj2} reduce to a sum over $z,z'\in Z$.
\end{lma}
\begin{proof}
  If $f(z)=f(z')$ and $f(z)$ is not an orbifold point then the contribution $k_{[z,z']}$ is a positive integer by Equation \eqref{eq:kzz}. Similarly, if $k_z\neq 0$ and $z\not\in Z$, then $k_z$ is a positive integer by Equation \eqref{eq:kz}. All other terms on the right-hand side of Equation \eqref{eq:adj2} are nonnegative and some of them are positive. The sum of this positive integer and these positive terms is supposed to be 1 (the left-hand side), which is impossible.
\end{proof}

\begin{lma}\label{lma:contribz}
  Suppose that $D\leq\Delta$. If $K_{z}=0$ and $f(z)=\orb{x}_k$ then $d_{z}=p_k^2$. As a consequence, the total contribution
  \[\frac{1}{2}\left(1-\frac{1}{d_{z}}\right)+k_z\]
  to the right-hand side of Equation \eqref{eq:adj2} from each point $z\in Z$ is either $\frac{1}{2}\left(1-\frac{1}{p_k^2}\right)$ (if $K_{z}=0$) or else it is greater than or equal to $1/2$.
\end{lma}
\begin{proof}
  If $d_{z}\neq p_k^2$ then there are $p_k^2/d_{z}$ lifts $\tilde{f}_\alpha$ of $f$ near $z$, which intersect pairwise at $0$. Hence $N_j\geq 1$ for $j=1,\ldots,p_k^2/d_{z}-1$, giving $K_z\geq\frac{p_k^2}{d_{z}}-1>0$. Therefore if $K_{z}=0$ then we have $d_{z}=p_k^2$.

  Each point $z\in Z$ contributes a total of $\frac{1}{2}\left(1+\frac{K_{z}-1}{d_{z}}\right)$ to the right-hand side of \eqref{eq:adj2}. This contribution is either $\frac{1}{2}\left(1-\frac{1}{p^2_k}\right)$ (if $K_{z}=0$) or else greater than or equal to $1/2$.  This implies the claim.
\end{proof}

\begin{lma}\label{lma:nopun}
  Suppose that $D\leq\Delta$. We have $Z=f^{-1}\left(\orb{X}_{\mathrm{sing}}\right)$.
\end{lma}
\begin{proof}
  By definition of an orbifold holomorphic map, an orbifold point of the curve must map to an orbifold point of $\orb{X}$, so $Z\subset f^{-1}\left(\orb{X}_{\mathrm{sing}}\right)$. For the reverse inclusion, note that if $f(z)=\orb{x}_k$ and $z\not\in Z$ then $d_z=1$. By Lemma \ref{lma:contribz}, this means $K_z\neq 0$ and by Lemma \ref{lma:kzzz} this cannot happen as $z\not\in Z$.
\end{proof}

\begin{lma}\label{lma:Z12}
  Suppose that $D\leq\Delta$. We have $1\leq |Z|\leq 2$.
\end{lma}
\begin{proof}
  If we had $|Z|=0$ then, by Lemma \ref{lma:nopun}, $f^{-1}\left(\orb{X}_{\mathrm{sing}}\right)$ would be empty and so the punctured curve $C$ would have no punctures. This would mean $\partial C=0\in H_1(\Sigma;\ZZ)$ and hence $D\equiv 0\mod \Delta^2$. Since $0<D\leq \Delta$, this is impossible.
  
  If we had $|Z|\geq 3$ then, by Lemma \ref{lma:contribz}, these three points together would contribute at least
  \[\sum_{z\in Z}\frac{1}{2}\left(1-\frac{1}{p_k^2}\right)\geq \frac{3}{2}-\frac{3}{2\min p_k^2}\geq\frac{3}{2}\left(1-\frac{1}{4}\right)=\frac{9}{8}>1\]
  to the right-hand side of Equation \eqref{eq:adj2} (since $p_k\geq 2$), which contradicts positivity of all the other terms.
\end{proof}

\begin{lma}
  Suppose that $D\leq\Delta$. For all $z,z'\in Z$ with $z\neq z'$ we have $f(z)\neq f(z')$, so that $k_{[z,z']}=0$.
\end{lma}
\begin{proof}
  By Lemma \ref{lma:kzzz}, if $k_{[z,z']}\neq 0$ then $z,z'\in Z$. If $|Z|=1$ then there is no $k_{[z,z']}$ term. Suppose that $|Z|=2$. We have that $K_{z}\geq 0$ with equality if and only if $d_{z}=p_k^2$ where $f(z)=\orb{x}_k$ (and the same for $z'$, where $f(z')=\orb{x}_{k'}$). Therefore Equation \eqref{eq:adj2} and Lemma \ref{lma:contribz} imply that
  \[0\geq \frac{(3\Delta-D)D}{\Delta^2}-\left(\frac{1}{p_k^2}+\frac{1}{p_{k'}^2}\right)+2k_{[z,z']}\]
  If both $f(z)=f(z')=\orb{x}_k$ then $k_{[z,z']}\geq\tfrac{1}{p_k^2}$ and the inequality reduces to
  \[0\geq (3\Delta-D)D.\]
  Since $0<D\leq \Delta$, we have $0<(3\Delta-D)D$ which is a contradiction. Therefore $f(z)\neq f(z')$ and there is no $k_{[z,z']}$ term.
\end{proof}

The adjunction formula for $\orb{C}$ is now
\[D^2-3\Delta D+(2-|Z|)\Delta^2=\sum_{z\in Z}(K_z-1)\frac{\Delta^2}{d_z}.\]
\begin{lma}\label{lma:coprime}
  Suppose that $D\leq\Delta$ and that either:
  \begin{itemize}
  \item $Z=\{z\}$, or
  \item $Z=\{z,z'\}$ and $K_{z'}=0$.
  \end{itemize}
  Then $K_z-1<d_z$. In particular, if $z\mapsto \left(z^Q,\sum_{i=1}^\infty a_iz^{R_i}\right)$ is a local parametrisation of one of the lifts $\tilde{f}_\alpha$ of $f$ to the local uniformising cover $B \subset \CC^2$ of a neighbourhood of $f(z)$, then $\gcd(Q,R_1)=1$.
\end{lma}
\begin{proof}
  Since $D^2-3\Delta D<0$, we have
  \[\sum_{z\in Z}(K_z-1)\frac{\Delta^2}{d_z}<(2-|Z|)\Delta^2.\]
  In the case $Z=\{z\}$, this gives $K_z-1<d_z$. In the second case, this gives
  \[(K_z-1)\frac{\Delta^2}{d_z}<\frac{\Delta^2}{d_{z'}}\]
  so $(K_z-1)/d_z<1$. In either case, Corollary \ref{cor:gcd} implies that $\gcd(Q,R_1)=1$.
\end{proof}

\begin{lma}\label{lma:Kzneq1}
  Suppose that $z\in Z$ is a point with $|G_z|=p^2$. Then $K_z\neq 1$.
\end{lma}
\begin{proof}
  Assume that $K_z=1$. Since $K_z=2N_0+N_1+\cdots+N_{(p^2/d_z)-1}$ with $N_i>0$ for all $i\neq 0$, we see that $N_0=0$, $N_1=1$ and $p^2/d_z=2$. Geometrically, this means that the lift of $f$ to the uniformising cover has two branches $\tilde{f}_\alpha$, $\alpha\in A=\{1,2\}$; moreover these branches are embedded at $z$ and intersect transversely there. Let $T_\alpha$ be the tangent line to $f_\alpha$ at $z$; every element of the deck group for the orbifold cover either switches of fixes the branches. Let $\zeta$ be an element of the deck group which switches the branches. Then $\zeta T_1=T_2$ and $\zeta T_2=T_1$, so the action of $\rt_{p^2}$ on the space of tangent complex lines at the origin has an orbit of size two. The action of $\rt_{p^2}$ on complex lines is the same as its action on Reeb orbits of $\Sigma_{p,q}$, so the orbits fall into two categories: exceptional orbits of size $1$ and generic orbits of size $p^2/g$ where $g=\gcd(p^2,pq-2)$ (see Section \ref{sct:contgeom}). Thus $p^2/g=2$ and, since $g\in\{1,2,4\}$ by Remark \ref{rmk:g124}, we must have $g=2$ and $p^2=4$. However, when $p=2$ we know that $g=4$, so we have a contradiction.
\end{proof}

We now analyse the cases $|Z|=1$ and $|Z|=2$ separately.

\subsection{The case $|Z|=1$}

Suppose that $Z=\{z\}$. The adjunction formula for $\orb{C}$ tells us that
\begin{equation}\label{eq:adjz1}D^2-3\Delta D+\Delta^2=(K_z-1)\frac{\Delta^2}{d_z}.\end{equation}

\begin{lma}\label{lma:oddfib}
  Suppose that $D\leq\Delta$ and $Z=\{z\}$. We have $f(z)=\orb{x}_N$ where $p_N$ is maximal in $\{p_1,\ldots,p_N\}$. Moreover, either:
  \begin{itemize}
  \item $K_z=0$, $d_z=p_N^2$ and $p_N$ is an odd-indexed Fibonacci number, or
  \item $K_z\neq 0$ and the homology class $[\orb{C}]=D\eE$ satisfies
    \[D<\frac{3-\sqrt{5}}{2}\Delta.\]
  \end{itemize}
  In either case, $D<\frac{2\Delta}{3}$.
\end{lma}
\begin{proof}
  If $f(z)=\orb{x}_k$ then the curve $\orb{C}$ does not intersect $\Sigma_i$ for $i\neq k$. Hence its homology class $[\orb{C}]=D\eE$ reduces to zero modulo $p_i^2$ for $i\neq k$. Therefore $D\geq \Delta^2/p_k^2$ which is greater than $\Delta$ if $p_k$ is not maximal. This would contradict the assumption $D\leq\Delta$. Therefore $f(z)=\orb{x}_N$ with $p_N$ maximal in $\{p_1,\ldots,p_N\}$.
  
  If $K_z=0$ then $d_z=p_N^2$ by Lemma \ref{lma:contribz}. The adjunction formula \eqref{eq:adjz1} tells us
  \[D^2+\Delta^2+\frac{\Delta^2}{p_N^2}=3\Delta D.\]
  Therefore $D=\frac{3p_N\pm\sqrt{5p_N^2-4}}{2}\frac{\Delta}{p_N}$. If this is an integer then $p_N$ is an odd-indexed Fibonacci number, $F_{2m+1}$. In this case, $\sqrt{5p_N^2-4}$ is an odd-indexed Lucas number\footnote{The Lucas numbers are defined by the Fibonacci recursion $L_n=L_{n-1}+L_{n-2}$ but with $L_0=2$, $L_1=1$.} and $D=p_1\cdots p_{N-1}F_{2m-1}$. Note: $D=\frac{3p_N-\sqrt{5p_N^2-4}}{2p_N}\Delta<\frac{2\Delta}{3}$.
  
  If $K_z\neq 0$ then $D^2-3\Delta D+\Delta^2\geq 0$. The quadratic function on the left hand side of this inequality is decreasing on the interval $[0,\Delta]$ where $D$ lives, so $D$ must be less than the root $\frac{3-\sqrt{5}}{2}\Delta$.
\end{proof}

\begin{lma}\label{lma:vdimbound}
  Suppose that $D\leq\Delta$ and $Z=\{z\}$. Then the virtual complex dimension of the moduli space containing $\orb{C}$ is less than or equal to 1. If it is equal to 1 then $D>\Delta/3$.
\end{lma}
\begin{proof}
  When $|Z|=1$, Equation \eqref{eq:vdim} for the virtual complex dimension of the moduli space reduces to:
  \[\frac{3D}{\Delta}-\frac{m_1+m_2}{|H_z|}.\]
  This is strictly less than $\frac{3D}{\Delta}$. By Lemma \ref{lma:oddfib}, we also know that $D<\frac{2\Delta}{3}$. Therefore the virtual dimension of the curves $\orb{C}$ is at most one. If the virtual dimension equals one then
  \[\frac{3D}{\Delta}-\frac{m_1+m_2}{|H_z|}=1,\]
  so $D>\Delta/3$ as required.
\end{proof}

\begin{lma}\label{lma:dsq}
  Suppose that $D\leq\Delta$, that $Z=\{z\}$ and that the virtual dimension of the moduli space containing $\orb{C}$ is equal to 1. Then
  \[\frac{D^2}{\Delta^2}=\frac{1}{p_N^2}\left(\frac{p_N^2}{d_z}\right)^2QR_1.\]
\end{lma}
\begin{proof}
  Let $\orb{C}$ and $\orb{C}'$ be two curves in this 1-dimensional moduli space. Since they both have homology class $D\eE$ and $\eE^2=1/\Delta^2$, we know that $D^2/\Delta^2$ is the sum of local intersection contributions from the points in $\orb{C}\cdot\orb{C}'$. We know from Lemma \ref{lma:oddfib} that $D<\frac{2\Delta}{3}$, so none of these intersection contributions can be integers as they sum to $D^2/\Delta^2<4/9$. Since the only orbifold point contained in both $\orb{C}$ and $\orb{C}'$ is $\orb{x}_N$, there is only one local intersection contribution, which we now compute.

  Both curves have $p_N^2/d_z$ lifts in the uniformising cover as they belong to the same moduli space. The local models for the lifts are of the form
  \[\tilde{f}_\alpha=\left(z^Q,\sum a_{\alpha,i}z^{R_i}\right)\qquad \tilde{f}'_\beta=\left(z^Q,\sum a'_{\beta,i}z^{R_i}\right),\]
  with $R_i\equiv R_1\mod d_z$. Lemma \ref{lma:jet} tells us that the local intersection contribution from each pair of branches is at least $QR_1$. There are $\left(\frac{p_N^2}{d_z}\right)^2$ pairs of branches and the sum of local intersection numbers is weighted by an overall factor of $1/p_N^2$ in Equation \eqref{eq:kzz}, so we get
  \[\frac{D^2}{\Delta^2}\geq\frac{1}{p_N^2}\left(\frac{p_N^2}{d_z}\right)^2QR_1.\]
  By Lemma \ref{lma:jet}, strict inequality means that one of the local intersection numbers $\tilde{f}_{\alpha}\cdot\tilde{f}'_\beta$ is strictly greater than $d_z$. This would mean that $\tilde{f}_{\zeta^i\alpha}\cdot\tilde{f}'_{\zeta^i\beta}>d_z$ for each $i=0,\ldots,(p_N^2/d_z)-1$ (that is, for the $|G_z/H_z|$ pairs of lifts which are conjugate to the pair $(\alpha,\beta)$ under the residual $G_z/H_z$-action). This would give
  \[\frac{D^2}{\Delta^2}>\frac{1}{p_N^2}\frac{p_N^2}{d_z}d_z=1,\]
  which is impossible, since $D\leq\Delta$. Therefore we have equality and the Lemma follows.
\end{proof}

\subsection{The case $|Z|=2$}

  Suppose that $Z=\{z,z'\}$. The adjunction formula for $\orb{C}$ tells us that
\begin{equation}\label{eq:adjz2}0=\frac{3\Delta D-D^2}{\Delta^2}+\frac{K_z-1}{d_z}+\frac{K_{z'}-1}{d_{z'}}.\end{equation}

\begin{lma}\label{lma:cyl}
  Suppose that $D\leq \Delta$, that $Z=\{z,z'\}$, and that $f(z)=\orb{x}_k$, $f(z')=\orb{x}_{k'}$ with $p_{k'}<p_k$. Then:
  \begin{enumerate}
    \item[A.] the term $K_{z'}$ vanishes and $d_{z'}=p_{k'}^2$;
    \item[B.] the homology class $D\eE$ of $\orb{C}$ satisfies $D<\Delta/3p_{k'}$;
    \item[C.] $p_k$ is maximal among $\{p_1,\ldots,p_N\}$, that is $k=N$;
    \item[D.] there is no other orbifold holomorphic curve $\orb{C}'$ with homology class $[\orb{C}']=D'\eE$ satisfying $D'\leq\Delta$ passing through the two orbifold points $\orb{x}_k$ and $\orb{x}_{k'}$.
  \end{enumerate}
\end{lma}
\begin{proof}
  \begin{enumerate}
  \item[A.] We must have either $K_z=0$ or $K_{z'}=0$: otherwise the last two terms on the right hand side of Equation \eqref{eq:adjz2} are nonnegative and cannot cancel the positive first term. Moreover, if $K_{z'}\neq 0$ then $K_z=0$ and hence, by Lemma \ref{lma:contribz}, $d_z=p_k^2$, so the adjunction formula \eqref{eq:adjz2} becomes
    \[0=\frac{3\Delta D-D^2}{\Delta^2}-\frac{1}{p_k^2}+\frac{K_{z'}-1}{d_{z'}}.\]
    Since $d_{z'}\leq p_{k'}^2<p_k^2$ and since $K_{z'}\neq 1$ by Lemma \ref{lma:Kzneq1}, the final term is strictly larger than $1/p_k^2$ and therefore the right hand side is still positive, which is a contradiction. Therefore $K_{z'}=0$, and $d_{z'}=p_{k'}^2$ by Lemma \ref{lma:contribz}.
  \item[B.] Using part A of the lemma, the adjunction formula \eqref{eq:adjz2} becomes
    \[D^2-3\Delta D+\frac{\Delta^2}{p_{k'}^2}=(K_z-1)\frac{\Delta^2}{d_z}\geq -\frac{\Delta^2}{p_k^2}.\]
    Since $0<D/\Delta\leq 1$, this means that $D/\Delta$ must lie in the interval $[0,D(p_{k'},p_{k})]$ where
    \[D(p_k,p_{k'}):=\frac{1}{2}\left(3-\sqrt{9-4\left(\frac{1}{p_{k'}^2}+\frac{1}{p_k^2}\right)}\right).\]
    On this interval, the function $H(x):=x^2-3x+\left(\frac{1}{p_{k'}^2}+\frac{1}{p_k^2}\right)$ is nonnegative. We compute that $H(1/3p_{k'})<0$, so $D(p_{k'},p_k)<\tfrac{1}{3p_{k'}}$ and hence $D<\tfrac{\Delta}{3p_{k'}}$ as required.
  \item[C.] Since $\orb{C}$ is disjoint from $\orb{x}_i$ for $i\neq k,k'$, the boundary $\partial C$ must vanish in $H_1(\bigcup_{i\neq k,k'}\Sigma_i;\ZZ)=\ZZ/(\Delta^2/p_{k'}^2p_k^2)$ and hence $\Delta^2/p_{k'}^2p_k^2$ divides $D$. Therefore
    \[\frac{\Delta^2}{p_{k'}^2p_k^2}\leq D<\frac{\Delta}{3p_{k'}}\]
    which means
    \[\frac{3\Delta}{p_{k'}p_k}\leq p_k.\]
    If $p_k$ is not maximal then $p_N$ divides the left hand side, which means that $p_N\leq p_k$, which contradicts maximality of $p_N$. Therefore $p_k$ is maximal.
  \item[D.] If there were two orbifold holomorphic curves $\orb{C}$ and $\orb{C}'$ in homology classes $D\eE$ and $D'\eE$ which both passed through the orbifold points $\orb{x}_k$ and $\orb{x}_{k'}$ then the intersection formula would give local contributions from these points of at least $\frac{1}{p_k^2}+\frac{1}{p_{k'}^2}$ to the product $\frac{DD'}{\Delta^2}$. Since part B of this Lemma implies that both $D$ and $D'$ are less than $\Delta/3p_{k'}$, this gives
    \[\frac{1}{p_k^2}+\frac{1}{p_{k'}^2}\leq \frac{DD'}{\Delta^2}\leq \frac{1}{9p_{k'}^2},\]
    which is a contradiction.
  \end{enumerate}
\end{proof}

\subsection{Analysis of SFT limit curves}\label{sct:sftanalysis}

Let $J\in\mathcal{J}_\Sigma$ and let $\orb{J}\in\mathcal{J}_{\orb{X}}$ be the associated orbifold almost complex structure on $\orb{X}$ given by Lemma \ref{lma:jhat}. By Lemma \ref{lma:adjusted}, any $J\in\mathcal{J}_\Sigma$ is an adjusted almost complex structure, which means we can perform neck-stretching along $\Sigma$ starting with $J$ and obtain a sequence of almost complex structures $J_t$ on $X$ for which the SFT compactness theorem applies. Pick points $w_1,w_2\in V$ and look at the unique $J_t$-holomorphic curve $u_t(w_1,w_2)$ in the class of a line $\cp{1}\subset\cp{2}$ passing through $w_1$ and $w_2$.

As $t\to\infty$, the sequence $u_t(w_1,w_2)$ converges (in the Gromov-Hofer sense, {\cite[Section 9]{BEHWZ}}) to a holomorphic building $u_\infty(w_1,w_2)$ in the split almost complex manifold made up of:
\begin{itemize}
\item $\overline{V}$, the completion of $V$, equipped with the almost complex structure $\overline{J|_V}$;
\item $\RR\times\Sigma$, the symplectisation of $\Sigma$, equipped with an complex structure making it biholomorphic to $\left(\CC^2\setminus\{(0,0)\}\right)/\Gamma_{p_i,q_i}$;
\item $\overline{B}$, the completion of $B$, equipped with the almost complex structure which is the completion of $J|_B$.
\end{itemize}

Let $\mathcal{C}=\{C_1,\ldots,C_n\}$ be an enumeration of the distinct simple punctured curves underlying those components of $u_\infty(w_1,w_2)$ which lie in the symplectic completion $\bar{V}$; let $\orb{\mathcal{C}}=\{\orb{C}_1,\ldots,\orb{C}_n\}$ denote the associated somewhere-injective orbifold $\orb{J}$-holomorphic curves in $\orb{X}$, constructed in Lemma \ref{lma:orbifoldcompactification}.

\begin{lma}\label{lma:area}
  For each $\orb{C}\in\orb{\mathcal{C}}$, if we write $[\orb{C}]=D\eE$ then we have $D\leq\Delta$. In particular, all of the conclusions from Section \ref{sct:orbicurve} apply to $\orb{C}$.
\end{lma}
\begin{proof}
  The symplectic area of $\cp{1}\subset\cp{2}$ is equal to the sum of the symplectic areas of the various components of the SFT limit building living in $\bar{V}$ {\cite[Corollary 2.11]{CieliebakMohnke}}. The curves $C\in\mathcal{C}$ are obtained from these components by taking the underlying simple curves and ignoring repeats, so the sum
  \[\sum_{C\in\mathcal{C}}\int_C\overline{\omega|_V}\]
  of symplectic areas is less than or equal to the area of $\cp{1}$. The class $H=[\cp{1}]\in H_2(X;\ZZ)$ is identified with the class $\Delta\eE\in H_2(V,\Sigma;\ZZ)$. Therefore if we normalise $H$ to have symplectic area $\Delta$, $\eE$ must have symplectic area $1$, and hence $D\eE$ has area $D$. This implies that $D\leq \Delta$.
\end{proof}

\begin{lma}
  For generic $J\in\mathcal{J}_\Sigma$ and generic $w_1,w_2\in V$ there exists a component $\orb{C}\in\orb{\mathcal{C}}$ with $|Z|=1$ such that the virtual (complex) dimension of the moduli space of $\orb{J}$-holomorphic orbifold curves containing $\orb{C}$ is 1.
\end{lma}
\begin{proof}
  Any nonconstant holomorphic curve in $\orb{X}$ necessarily passes through $V$, so if we pick $J$ generically on $V$ then, for any fixed $E>0$, we can achieve transversality for all somewhere-injective orbifold $\orb{J}$-holomorphic curves with energy at most $E$. In particular, we can achieve transversality for all curves with $|Z|=1$, virtual complex dimension zero living in a homology class $m\eE$ with $m\leq\Delta$. In particular, there is a finite collection of such curves. By Lemma \ref{lma:cyl} D, there is at most a finite number of somewhere-injective curves with energy less than $\Delta$ and $|Z|=2$ and, by Lemma \ref{lma:Z12}, there are no curves with $|Z|\geq 3$.

  We pick $w_1,w_2$ in $V$ so that they do not lie on any of the finite set of curves with $|Z|\neq 1$ or virtual dimension zero. Since the condition of containing $w_i$ is closed, one of the curves $\orb{C}\in\orb{\mathcal{C}}$ contains $w_1$, and hence (by our choice of $w_1$) it must have $|Z|=1$ and virtual dimension at least 1. By Lemma \ref{lma:vdimbound}, its virtual dimension is precisely 1.
\end{proof}

\begin{lma}
  If $N\geq 2$ then, for any $j\neq N$, the set $\orb{\mathcal{C}}$ contains a unique curve $\orb{C}_{\mathrm{cyl},j}$ through $\orb{x}_j$. This curve has $|Z|=2$ and $\orb{x}_N\in\orb{C}_{\mathrm{cyl},j}$.
\end{lma}
\begin{proof}
  Fix $j$. Each $J_t$-holomorphic curve in the class of a line $\cp{1}\subset\cp{2}$ necessarily passes through the pinwheel $L_{p_j,q_j}\subset B_j$. This is because we can trivialise the complex determinant line bundle of $\cp{2}$ in the complement of $\cp{1}$, but any neighbourhood of $L_{p_j,q_j}$ has nontrivial first Chern class. Since the condition that a curve intersects $L_{p_j,q_j}$ is a closed condition, the SFT limit building must contain a component which intersects $L_{p_j,q_j}$. This component lives in the symplectic completion of the rational homology ball $\bar{B}_j$, so, correspondingly, there must be a component of the building in $\bar{V}$ with punctures asymptotic to Reeb orbits in $\Sigma_j$. After taking the orbifold compactification, this means that there is a component of the orbifold curve containing $\orb{x}_j$. A curve containing only $\orb{x}_j$ would necessarily live in a homology class $D_j\eE$ with $\Delta^2/p_j^2$ dividing $D_j$, but this is strictly bigger than $\Delta$ if $j\neq N$. Hence, by Lemma \ref{lma:area}, this cannot occur. Therefore this component has $|Z|=2$. By Lemma \ref{lma:cyl} C, this curve also contains $\orb{x}_N$. Uniqueness follows from Lemma \ref{lma:cyl} D.
\end{proof}

\begin{thm}\label{thm:oneball}
  Suppose that $N=1$. Let $\orb{C}\in\orb{\mathcal{C}}$ be a component with $Z=\{z\}$ living in a moduli space of virtual dimension 1. Then we have $d_z=p_1^2$. If $z\mapsto (z^Q,a_1z^{R_1}+\cdots)$ is a local model for the lift of the curve to a uniformising cover of a neighbourhood of the orbifold point $\orb{x}_1$ then there exist positive integers $b,c<p_1$ such that $Q=b^2$, $R_1=c^2$ and
  \[p_1^2+b^2+c^2=3p_1bc.\]
  Moreover, $q_1$ is determined up to $q_1\mapsto p_1-q_1$ by
  \[bq_1=\pm 3c\mod p_1.\]
\end{thm}
\begin{proof}  
  Let $z\mapsto \left(z^Q,\sum a_iz^{R_i}\right)$ be a local model for one of the lifts $\tilde{f}_\alpha$ to the uniformising cover. By Lemma \ref{lma:dsq}, we have $D^2=\left(\frac{p_1^2}{d_z}\right)^2QR_1$, and, by Lemma \ref{lma:coprime}, we have $\gcd(Q,R_1)=1$. This implies that $Q=b^2$ and $R_1=c^2$ for some integers $b,c$. Since $bc$ divides $D^2<p_1^2$, we have $b,c<p_1$. Recall the adjunction formula \eqref{eq:adjz1} in this case is
  \[D^2-3\Delta D+\Delta^2=(K_z-1)\frac{\Delta^2}{d_z},\]
  where now: $\Delta=p_1$, $D=\tfrac{p_1^2}{d_z}bc$, and (by Corollary \ref{cor:gcd})
  \begin{align*}
    K_z-1&=\frac{p_1^2}{d_z}QR_1-Q-R_1\\
    &=\frac{p_1^2}{d_z}b^2c^2-b^2-c^2.
  \end{align*}
  This becomes
  \[\left(\frac{p_1^2}{d_z}\right)^2b^2c^2-3\frac{p_1^3bc}{d_z}+p_1^2=\left(\frac{p_1^2}{d_z}b^2c^2-b^2-c^2\right)\frac{p_1^2}{d_z},\]
  or
  \[d_z+b^2+c^2=3p_1bc.\]
  Let us factorise $d_z$ as $d_z=s^2t$ where $t$ is squarefree. Since $d_z$ divides $p_1^2$, we see that $p_1=str$ for some $r$. Multiplying by $r^2$ we get
  \[r^2s^2t+(rb)^2+(rc)^2=3rst(rb)(rc)\]
  so that $rb,rc,rs$ is a positive integer solution to the Diophantine equation
  \[x^2+y^2+tz^2=3txyz.\]
  There is a complete classification, due to Rosenberger \cite{Rosenberger79}, of Diophantine equations of the form $\alpha x^2+\beta y^2+\gamma z^2=\delta xyz$, with $1\leq\alpha\leq\beta\leq\gamma$, $\gcd(\alpha,\beta)=\gcd(\alpha,\gamma)=\gcd(\beta,\gamma)$ and with $\delta$ divisible by $\alpha\beta\gamma$, which admit positive integer solutions. They are called {\em Markov-Rosenberger equations} and there are six possibilities:
  \[(\alpha,\beta,\gamma,\delta)=\begin{cases}
    (1,1,1,1)\\ (1,1,1,3)\\ (1,1,2,2)\\ (1,1,2,4)\\ (1,2,3,6)\\ (1,1,5,5).
  \end{cases}\]
  The equation $x^2+y^2+tz^2=3txyz$ therefore has positive integer solutions if and only if $t=1$. But then $(rb,rc,rs)$ is a solution of the Markov equation $x^2+y^2+z^2=3xyz$ so $r=\gcd(x,y,z)=1$. Therefore $d_z=s^2=p_1^2$ and
  \[p_1^2+b^2+c^2=3p_1bc.\]
  Finally, by Remark \ref{rmk:vdim}, we know that $R_1=Q(\pm p_1q_1-1)\mod p_1^2$. Since $R_1=c^2$ and $Q=b^2$ this tells us that
  \[c^2=b^2(\pm p_1q_1)\mod p_1^2.\]
  Since $b^2+c^2=(3bc-p_1)p_1$ this gives
  \[3bc=\pm b^2q_1\mod p_1\]
  or $3c=\pm bq_1\mod p_1$ (since $b$ and $p_1$ are coprime).
\end{proof}

\begin{thm}\label{thm:twoball}
  Suppose that $N=2$. Let $\orb{C}=\orb{C}_{\mathrm{cyl},1}\in\orb{\mathcal{C}}$ be the unique component with $Z=\{z,z'\}$, $f(z')=\orb{x}_1$, $f(z)=\orb{x}_2$, and let $[\orb{C}]=D\eE$. Then
  \[p_1^2+p_2^2+D^2=3p_1p_2D.\]
  Moreover $3D=\pm p_1q_2\mod p_2$ and $3D=\pm p_2q_1\mod p_1$.
\end{thm}
\begin{proof}
  By Lemma \ref{lma:cyl} A, we know that $K_{z'}=0$ and $d_{z'}=p_1^2$, so that the adjunction formula \eqref{eq:adjz2} becomes
  \begin{equation}\label{eq:adjtwoball}D^2-3p_1p_2D+p_2^2=(K_z-1)\frac{p_2^2}{d_z}p_1^2.\end{equation}
  The virtual dimension of the moduli space containing $\orb{C}$ is zero since $\orb{C}$ is the unique curve in its homology class (we can choose $J$ generically on $V$ so that $\orb{C}$ is regular). The virtual dimension formula, Equation \eqref{eq:vdim}, becomes
  \begin{equation}\label{eq:vdimtwoball}\frac{3D}{p_1p_2}+1=\frac{m_{1,z}+m_{2,z}}{d_z}+\frac{m_{1,z'}+m_{2,z'}}{p_1^2}.\end{equation}
  To find the numbers $m_{i,z}$, $m_{i,z'}$, we use Remark \ref{rmk:vdim}: if $\tilde{f}_\alpha(z)=\left(z^Q,\sum a_iz^{R_i}\right)$ is a local lift of $\orb{C}$ in a neighbourhood of $\orb{x}_2$ then $m_{1,z}=[Q]_{d_z}$ and $m_{2,z}=R_1=[Q(\pm p_2q_2-1)]_{d_z}$ where we write $[x]_y$ for the remainder of $x$ modulo $y$. Similarly we get $m_{1,z'}=1$ and $m_{2,z'}=[\pm p_1q_1-1]_{p_1^2}$ because the condition $K_{z'}=0$ means that, in local coordinates, the lift of $\orb{C}$ in a neighbourhood of $\orb{x}_1$ is $\left(z,\sum a'_iz^{R'_i}\right)$.

  Note that by Lemma \ref{lma:coprime}, we know that $\gcd(Q,R_1)=1$, so by Corollary \ref{cor:gcd} we have $K_z-1=\frac{p_2^2}{d_z}QR_1-Q-R_1$.
  
  We have
  \begin{align*}
    p_1p_2\frac{m_{1,z'}+m_{2,z'}}{p_1^2}&=p_1p_2\frac{1+[\pm p_1q_1-1]_{p_1^2}}{p_1^2}\\
                                         &=p_2\frac{[\pm p_1q_1]_{p_1^2}}{p_1}\\
                                         &=p_2[\pm q_1]_{p_1}.
  \end{align*}
  The sign ambiguity here is simply a question of which plane $\CC\times\{0\}$ or $\{0\}\times\CC$ our curve asymptotes; since $q_1$ is only determined up to a sign modulo $p_1$, we can relabel the $q_1$ we have picked to be $p_1-q_1$ if necessary in order to replace the term $[\pm q_1]_{p_1}$ by $q_1$ (this relabelling does not affect any other equations as $q_1$ appears only here). The virtual dimension formula \eqref{eq:vdimtwoball} now implies
  \[3D+p_1p_2=p_1(Q+R_1)\frac{p_2^2}{d_z}\frac{1}{p_2}+p_2q_1.\]

We now define $\tilde{Q}=\frac{p_2^2}{d_z}Q$ and $\tilde{R}=\frac{p_2^2}{d_z}R_1$, so that
  \begin{equation}\label{eq:vdim2}3D+p_1p_2=p_1(\tilde{Q}+\tilde{R})\frac{1}{p_2}+p_2q_1.\end{equation}
  In terms of $\tilde{Q}$ and $\tilde{R}$, we have $(K_z-1)p_2^2/d_z=\tilde{Q}\tilde{R}-\tilde{Q}-\tilde{R}$, and the adjunction formula \eqref{eq:adjtwoball} becomes
  \begin{equation}\label{eq:adjcyl}D^2-3p_1p_2D+p_2^2=(\tilde{Q}\tilde{R}-\tilde{Q}-\tilde{R})p_1^2.\end{equation}
  \begin{clm}
    We have $\tilde{Q}+\tilde{R}=[\pm\tilde{Q}p_2q_2]_{p_2^2}$.
  \end{clm}
  \begin{proof}
    We know that
    \begin{align*}
      \tilde{Q}+\tilde{R}&=\tilde{Q}+\tilde{Q}(\pm p_2q_2-1)\mod p_2^2\\
      &=\pm\tilde{Q}p_2q_2\mod p_2^2,
      \end{align*}
    so it suffices to show that $\tilde{Q}+\tilde{R}<p_2^2$. Equation \eqref{eq:vdim2} tells us that
    \[\frac{3p_2D}{p_1}+p_2^2\left(1-\frac{q_1}{p_1}\right)=\tilde{Q}+\tilde{R}.\]
    By Lemma \ref{lma:cyl} B, we have $D<\frac{p_2}{3}$, so
    \[\tilde{Q}+\tilde{R}<p_2^2\left(1-\frac{q_1}{p_1}+\frac{1}{p_1}\right).\]
    This implies that $\tilde{Q}+\tilde{R}$ is strictly less than $p_2^2$ as $q_1\geq 1$.
  \end{proof}
  \begin{clm}
    We have $\tilde{Q}<\frac{2p_2}{p_1}$.
  \end{clm}
  \begin{proof}
    We know that $D^2-3p_1p_2D<0$, so
    \begin{align*}
      p_2^2&>D^2-3p_1p_2D+p_2^2\\
      &>((\tilde{Q}-1)(\tilde{R}-1)-1)p_1^2\\
      &>(\tilde{Q}^2-\tilde{Q}-1)p_1^2
    \end{align*}
    where we have used Equation \eqref{eq:adjcyl} and the fact that $\tilde{R}\geq\tilde{Q}+1$ (see the definition of the local model in Remark \ref{rmk:locmod}). If we define $f(x)=x^2-x-1$ then this tells us $f(\tilde{Q})<\tfrac{p_2^2}{p_1^2}$. We will show that $f\left(\tfrac{2p_2}{p_1}\right)>\tfrac{p_2^2}{p_1^2}$; since the function $f(x)$ is monotonically increasing on the interval $[1/2,\infty)$, this will imply the claim. We have
      \begin{align*}
        f\left(\frac{2p_2}{p_1}\right)&=4\frac{p_2^2}{p_1^2}-2\frac{p_2}{p_1}-1\\
        &=\frac{p_2^2}{p_1^2}+\left(\frac{3p_2}{p_1}-2\right)\frac{p_2}{p_1}-1\\
        &>\frac{p_2^2}{p_1^2}
      \end{align*}
      where we have used $p_2/p_1>1$.
  \end{proof}
  Let us define $k:=[\pm\tilde{Q}q_2]_{p_2}$ (the sign choice is the one for which $\tilde{Q}+\tilde{R}=kp_2$). In terms of $k$, the adjunction formula \eqref{eq:adjcyl} and virtual dimension formula \eqref{eq:vdim2} become
  \begin{align}
    \label{eq:adjpen} D^2-3p_1p_2D+p_2^2&=\left(\tilde{Q}(kp_2-\tilde{Q})-kp_2\right)p_1^2,\\
    \label{eq:vlast} p_1k&=3D+p_1p_2-p_2q_1,
  \end{align}
  so
  \begin{align}
    \label{eq:adjlast} D^2+(p_1\tilde{Q})^2+p_2^2&=3p_1p_2D+p_1p_2(p_1k)(\tilde{Q}-1)\\
    \nonumber &=3(p_1\tilde{Q})p_2D+p_1(p_1-q_1)p_2^2(\tilde{Q}-1).
  \end{align}
  We know that $D<p_2$ and $\tilde{Q}<2p_2/p_1$, so
  \[D^2+(p_1\tilde{Q})^2<p_2D+2p_2p_1\tilde{Q}<3(p_1\tilde{Q})p_2D.\]
  Therefore,
  \[p_2^2>p_1(p_1-q_1)p_2^2(\tilde{Q}-1).\]
  However, the right-hand side is strictly bigger than $p_2^2$ unless $\tilde{Q}=1$. Thus we deduce that $\frac{p_2^2}{d_z}Q=1$, so $d_z=p_2^2$ and $Q=1$. The adjunction formula \eqref{eq:adjlast} now gives
  \[D^2+p_1^2+p_2^2=3p_1p_2D,\]
  as required.

  The equalities $3D=\pm p_1q_2\mod p_2$ and $3D=\pm p_2q_1\mod p_1$ follow from the virtual dimension formula \eqref{eq:vlast} (recall from Remark \ref{rmk:signamb} that $q_i$ is only uniquely determined up to a sign modulo $p_i$).
\end{proof}

\begin{rmk}\label{rmk:unicity}
  Note that, for any given Markov number $p$ and for any $q$, there is at most one Markov triple $(a,b,p)$ with $a,b<p$ and $aq=\pm 3b\mod p$. See {\cite[Proposition 3.15]{Aigner13}}.
\end{rmk}

\begin{thm}\label{thm:threeball}
  Suppose that $N=3$. Then $(p_1,p_2,p_3)$ is a Markov triple.
\end{thm}
\begin{proof}
  Suppose without loss of generality that $p_3>p_1,p_2$. By Theorem \ref{thm:twoball}, there exist $b_1<p_3$ and $b_2<p_3$ such that
  \begin{align*}
    b_1^2+p_2^2+p_3^2&=3b_1p_2p_3,\\
    p_1^2+b_2^2+p_3^2&=3p_1b_2p_3,\\
    p_2q_3&=\pm 3b_1\mod p_3,\\
    p_1q_3&=\pm 3b_2\mod p_3.
  \end{align*}
  By Remark \ref{rmk:unicity}, this implies that $b_1=p_1$ and $b_2=p_2$. Hence, $(p_1,p_2,p_3)$ is a Markov triple.
\end{proof}

\begin{cor}\label{cor:fourball}
  We have $N\leq 3$.
\end{cor}
\begin{proof}
  If $(a,b,c)$ and $(a,b,c')$ are both Markov triples then $c'=3ab-c$ and we say that $(a,b,c)$ and $(a,b,c')$ are related by a mutation. The graph whose vertices are Markov triples and whose edges are mutations forms a tree called the {\em Markov tree} (see for example {\cite[Theorem 3.3]{Aigner13}} for a proof that this is a tree). Suppose $N\geq 4$. Then, by Theorem \ref{thm:threeball}, $(p_i,p_j,p_k)$ is a Markov triple for any $\{i,j,k\}\subset\{1,2,3,4\}$. This would imply that the Markov tree contains a copy of the complete graph on four vertices, which is a contradiction.
\end{proof}

\subsection*{Acknowledgements}

The authors would like to thank Georgios Dimitroglou Rizell, Paul Hacking, Yank{\i} Lekili, Oscar Randal-Williams and Chris Wendl for useful conversations and comments, and to thank Renato Vianna, Emily Maw and an anonymous referee for their many detailed corrections to earlier versions of this manuscript.

I.S. is partially supported by a Fellowship from the EPSRC.

\bibliographystyle{plain}
\bibliography{pinwheel}

\end{document}